\documentclass[a4paper,12pt]{article}
\usepackage{amsmath}
\usepackage{amssymb}
\usepackage{amsthm}
\usepackage{wrapfig}
\usepackage{comment}
\usepackage{braket}
\usepackage{mathrsfs}
\usepackage{cases}
\numberwithin{equation}{section}
\usepackage[top=35mm,bottom=35mm,left=25.4mm,right=25.4mm]{geometry}
\usepackage[dvipdfmx]{graphicx}
\usepackage{tikz}
\usepackage[abbrev]{amsrefs}
\usetikzlibrary{decorations.markings}

\theoremstyle{definition}
\newtheorem{df}{Definition}[section]
\newtheorem{lem}[df]{Lemma}
\newtheorem{prop}[df]{Proposition}

\newtheorem{rmk}[df]{Remark}
\newtheorem{thm}[df]{Theorem}

\def\R{\mathbb{R}}

\def\I{\mathcal{I}}
\def\J{\mathcal{J}}
\def\K{\mathcal{K}}

\def\T{\mathcal{T}}

\def\D{\mathscr{D}}

\def\O{\mathscr{O}}
\def\div{{\rm{div}}\,}

\begin{document}
\title{\bf{Attainability of a stationary Navier-Stokes flow 
around a rigid body rotating from rest}}
\author{By\\Tomoki T\small{AKAHASHI}\\(Nagoya University, Japan)}
\date{}
\maketitle

\noindent{\bf Abstract.}~ We consider the large time 
behavior of 
the three-dimensional 
Navier-Stokes flow around a rotating rigid body. 
Assume that the angular velocity of the body gradually increases 
until it reaches a small terminal one at a certain finite time and 
it is fixed afterwards. 
We then show that the fluid motion converges 
to a steady solution as time $t\rightarrow \infty$. 

\section{Introduction}\label{intro}
~~~We consider the large time behavior of a 
viscous incompressible flow 
around a rotating rigid body in $\R^3$.  
Assume that both a compact rigid body $\O$ 
and a viscous incompressible fluid 
which occupies the outside of $\O$ are initially at rest; then, 
the body starts to rotate with the angular velocity
which gradually increases until it reaches a small terminal one 
at a certain finite time and it is fixed afterwards. 
We then show that the fluid motion converges 
to a steady solution obtained by Galdi \cite{galdi2003} 
as time $t\rightarrow \infty$ 
(Theorem \ref{thm1} in Subsection \ref{maintheorem2}). 
This was conjectured by Hishida \cite[Section 6]{hishida2013}, 
but it has remained open. 
Such a question is called the starting problem and 
it was originally raised by 
Finn \cite{finn1965}, in which rotation was 
replaced by translation of the body.  
Finn's starting problem was first 
studied by Heywood \cite{heywood1972}; since his paper, 
a stationary solution is said to be attainable 
if the fluid motion converges to it as $t\rightarrow \infty$.  
Later on, by using Kato's approach \cite{kato1984} 
(see also Fujita and Kato \cite{fuka1964}) together with  
the $L^q$-$L^r$ estimates for the 
Oseen equation established by Kobayashi and Shibata \cite{kosh1998}, 
Finn's starting problem was completely solved 
by Galdi, Heywood and Shibata \cite{gahesh1997}. 
\par Let us introduce the mathematical formulation. 
Let $\O\subset \R^3$ be a compact and connected set 
with non-empty interior. 
The motion of $\O$ mentioned above is described 
in terms of the angular velocity 
\begin{align}
\omega(t)=\psi(t)\omega_0,\quad \omega_0=(0,0,a)^{\top}
\end{align}
with a constant $a\in \R$, where $\psi$ is a 
function on $\R$ satisfying the following conditions:
\begin{align}\label{psidef}
\psi\in C^1(\R;\R),\quad 
|\psi(t)|\leq 1~~{\rm{for}}~t\in\R,
~~\psi(t)=0~~{\rm{for}}~t\leq 0,~~
\psi(t)=1~~{\rm{for}} ~t\geq 1. 
\end{align}
Here and hereafter, $(\cdot)^\top$ denotes the transpose. 
Then the domain occupied by the fluid can be expressed as 
$D(t)=\{y=O(t)x;\,x\in D\}$, where 
$D=\R^3\setminus \O$ is assumed to be an exterior domain with smooth 
boundary $\partial D$ and 
\begin{align*}  
O(t)=\left(
\begin{array}{ccc}
\cos\Psi(t)&-\sin\Psi(t)&0\\
\sin\Psi(t)&\cos\Psi(t)&0\\
0&0&1
\end{array}\right),\quad \Psi(t)=\int_0^t \psi(s)a\,ds.
\end{align*}
We consider the initial boundary value problem 
for the Navier-Stokes equation
\begin{align}\label{NS1}
\left\{
\begin{array}{r@{}c@{}ll}
\partial_t w+w\cdot \nabla_y w
&{}={}&\Delta_y w-\nabla_y \pi,&y\in D(t),t>0,\\
\nabla_y\cdot w&{}={}&0,&y\in D(t),t\geq 0,\\
w|_{\partial D(t)}&{}={}&\psi(t)\omega_0\times y,&t\geq 0,\\
w(y,t)&{}\rightarrow{}&0&{\rm{as}}~|y|\rightarrow \infty,\\
w(y,0)&{}={}&0,&y\in D,
\end{array}\right.
\end{align}
where $w=(w_1(y,t),w_2(y,t),w_3(y,t))^\top$ 
and $\pi=\pi(y,t)$ denote  unknown velocity and 
pressure of the fluid, respectively. 
To reduce the problem to an equivalent one 
in the fixed domain $D$, 
we take the frame $x=O(t)^\top y$ attached to the body and make the 
change of the unknown functions:
$u(x,t)=O(t)^\top w(y,t),~p(x,t)=\pi(y,t)$. 
Then the problem (\ref{NS1}) is reduced to
\begin{align}\label{NS2}
\left\{
\begin{array}{r@{}c@{}ll}
\partial_t u+u\cdot \nabla u
&{}={}&\Delta u+(\psi(t)\omega_0\times x)\cdot \nabla u
-\psi(t)\omega_0\times u-\nabla p,
&x\in D,t>0,\\
\nabla\cdot u&{}={}&0,\hspace{2cm} x\in D,t\geq 0,\\
u|_{\partial D}&{}={}&\psi(t)\omega_0\times x,\hspace{0.3cm}t\geq 0,\\
u&{}\rightarrow{}&0\hspace{2.2cm}{\rm{as}}~|x|\rightarrow \infty,\\
u(x,0)&{}={}&0,\hspace{2cm} x\in D.
\end{array}\right.
\end{align}
\par The purpose of this paper 
is to show that (\ref{NS2}) 
admits a global solution which tends to a solution $u_s$ 
for the stationary problem
\begin{align}\label{sta}
\left\{
\begin{array}{r@{}c@{}ll}
u_s\cdot \nabla u_s&{}={}&\Delta u_s+(\omega_0\times x)\cdot \nabla u_s
-\omega_0\times u_s-\nabla p_s,&x\in D,\\
\nabla\cdot u_s&{}={}&0,\hspace{1.3cm}x\in D,\\
u_s|_{\partial D}&{}={}&\omega_0\times x,\\
u_s&{}\rightarrow{}&0\hspace{1.5cm}{\rm{as}}~|x|\rightarrow \infty.
\end{array}\right.
\end{align}
The rate of convergence in $L^r$ with 
$r\in (3,\infty]$ is also studied. 
In \cite{galdi2003}, Galdi successfully proved that 
if $|\omega_0|$ is sufficiently small, 
problem (\ref{sta}) has a 
unique smooth solution $(u_s,p_s)$ with pointwise estimates
\begin{align}\label{pointwise}
|u_s(x)|\leq \frac{C|\omega_0|}{|x|},\quad |\nabla u_s(x)|+|p_s(x)|
\leq \frac{C|\omega_0|}{|x|^2}.
\end{align}
We note that the decay rate (\ref{pointwise}) 
is scale-critical, which is also captured in terms of 
the Lorentz space (weak-Lebesgue space) $L^{3,\infty}$. 
This was in fact done by Farwig and Hishida \cite{fahi2007} 
even for the external force being in 
a Lorentz-Sobolev space of order $(-1).$ 
\par Let us mention some difficulties of our problem 
and how to overcome them in this paper. 
In \cite{gahesh1997}, 
the $L^q$-$L^r$ estimates for the Oseen semigroup play an important role. 
In the rotational case with constant angular velocity, 
Hishida and Shibata \cite{hish2009} 
also established the $L^q$-$L^r$ estimates of 
the semigroup generated by the Stokes operator 
with the additional term $(\omega_0\times x)\cdot \nabla-\omega_0\times$. 
If we use this semigroup as in \cite{gahesh1997}, 
we have to treat the term 
$(\psi(t)-1)(\omega_0\times x)\cdot\nabla v$, which is however 
no longer perturbation from the semigroup on account of 
the unbounded coefficient $\omega_0\times x$, where $v=u-\psi(t)u_s$. 
In this paper, 
we make use of the evolution operator $\{T(t,s)\}_{t\geq s\geq 0}$ on 
the solenoidal space  
$L^q_{\sigma}(D)~(1<q<\infty$), 
which is a solution operator to the linear problem 
\begin{align}\label{EV1}
\left\{
\begin{array}{r@{}c@{}ll}
\partial_t u&{}={}&\Delta u+(\psi(t)\omega_0\times x)\cdot 
\nabla u-\psi(t)\omega_0\times u-\nabla p,&x\in D,t>s,\\
\nabla\cdot u&{}={}&0,\hspace{0.3cm} x\in D,t\geq s,\\
u|_{\partial D}&{}={}&0,\hspace{0.3cm} t> s,\\
u&{}\rightarrow{}&0\hspace{0.45cm}{\rm{as}}~|x|\rightarrow \infty,\\
u(x,s)&{}={}&f,\hspace{0.3cm} x\in D.
\end{array}\right.
\end{align}
Hansel and Rhandi \cite{harh2014} succeeded in the proof of generation of 
this evolution operator with the $L^q$-$L^r$ smoothing rate. 
They constructed the evolution operator in their own way since the 
corresponding semigroup is not analytic 
(Hishida \cite{hishida1999}, Farwig and Neustupa \cite{fane2007}). 
Recently, Hishida \cite{hishida2018,hishidapre} 
developed the $L^q$-$L^r$ decay estimates of the 
evolution operator, see Section 3. With those estimates, 
we solve the integral equation which perturbation 
from the stationary solution $u_s$ obeys. 
However, it is difficult to perform analysis 
with the standard Lebesgue space 
on account of the scale-critical pointwise estimates (\ref{pointwise}).  
Thus, we first construct a solution for the weak formulation  
in the framework of Lorentz 
space by the strategy due to Yamazaki \cite{yamazaki2000}. 
We next identify this solution with a local solution 
possessing better regularity in a neighborhood of each time. 
The later procedure is actually adopted 
by Kozono and Yamazaki \cite{koya1998}. 
Furthermore, we derive the $L^{\infty}$ decay which is not observed 
in \cite{gahesh1997}. 
When the stationary solution possesses 
the scale-critical rate $O(1/|x|)$,   
Koba \cite{koba2017} first derived the $L^{\infty}$ decay 
of perturbation with less sharp rate in 
the context of stability analysis, see also Remark \ref{rmk3}. 
Although he used both the $L^1$-$L^r$ estimates of 
the Oseen semigroup $T(t)$ and 
the $L^q$-$L^r$ estimates (yielding the $L^q$-$L^\infty$ estimates) 
of the composite 
operator $T(t)P\div$, where $P$ denotes the Fujita-Kato projection 
(see Subsection \ref{notation}),  
it turns out that either of them is enough to accomplish the proof. 
In this paper, we employ merely the $L^1$-$L^r$ estimates of the 
adjoint evolution operator $T(t,s)^*$ to simplify the argument. 
\par The paper is organized as follows. 
In Section 2 we introduce the notation and give the main theorems. 
Section 3 is devoted to some preliminary results 
on the stationary problem and the evolution operator. 
In Section 4 we give the proof of the main theorems.

\section{Main theorems} \label{maintheorem}
~~~In this section, we first introduce some notation and 
after that, we give our main theorems.
\subsection{Notation}\label{notation}
~~~We introduce some function spaces. 
Let $D\subset \R^3$ be an exterior domain with smooth boundary. 
By $C^{\infty}_0(D)$, we  
denote the set of all $C^{\infty}$ functions 
with compact support in $D$. For $1\leq q\leq \infty$    
and nonnegative integer $m$, $L^q(D)$ and 
$W^{m,q}(D)$ denote the standard Lebesgue and Sobolev spaces, respectively. 
We write the $L^q$ norm as $\|\cdot\|_q$. 
The completion of $C_0^\infty(D)$ in $W^{m,q}(D)$ is 
denoted by $W_0^{m,q}(D)$. 
Let $1<q<\infty$ and $1\leq r\leq \infty$. 
Then the Lorentz spaces $L^{q,r}(D)$ are defined by
\begin{align*}
L^{q,r}(D)=\{f:{\rm{Lebesgue ~measurable~ function}}
\mid\|f\|^*_{q,r}<\infty\},
\end{align*}
where                
\begin{align*}
\|f\|^*_{q,r}=
\begin{cases}
\left(\displaystyle\int_0^{\infty}\big(t
\mu(\{x\in D\mid|f(x)|>t\})^{\frac{1}{q}}\big)^r\frac{dt}{t}\right)
^{\frac{1}{r}}
&1\leq r<\infty,\\
\displaystyle\sup_{t>0}
t\mu(\{x\in D\mid|f(x)|>t\})^{\frac{1}{q}}&r=\infty
\end{cases}
\end{align*}
and $\mu(\cdot)$ denotes the Lebesgue measure on $\R^3$. 
The space $L^{q,r}(D)$ is a 
quasi-normed space and it is even 
a Banach space equipped with 
norm $\|\cdot\|_{q,r}$ equivalent to $\|\cdot\|^*_{q,r}$. 
The real interpolation functor 
is denoted by $(\cdot,\cdot)_{\theta,r}$, 
then we have 
\begin{align*}
L^{q,r}(D)=\left(L^{q_0}(D),L^{q_1}(D)\right)_{\theta,r},
\end{align*}
where $1\leq q_0<q<q_1\leq \infty$ and $0<\theta<1$ 
satisfy $1/q=(1-\theta)/q_0+\theta/q_1$, while 
$1\leq r\leq \infty$, see Bergh-L\"{o}fstr\"{o}m \cite{belobook1976}. 
We note that if $1\leq r<\infty$, the dual of the space 
$L^{q,r}(D)$ is $L^{q/(q-1),r/(r-1)}(D)$. 
It is well known that if $1\leq r<\infty$,
the space $C^{\infty}_{0}(D)$ is dense in $L^{q,r}(D)$, while 
the space $C^{\infty}_{0}(D)$ is not dense in $L^{q,\infty}(D)$. 
\par We next introduce some solenoidal function spaces. 
Let $C_{0,\sigma}^\infty(D)$ be the 
set of all $C_0^{\infty}$-vector fields $f$ which satisfy $\div f=0$ 
in $D$. For $1<q<\infty$, 
$L^{q}_{\sigma}(D)$ denote the 
completion of $C_{0,\sigma}^\infty(D)$ in $L^q(D)$. 
For every $1<q<\infty$, we have the following Helmholtz decomposition:
\begin{align*}
L^q(D)=L^q_{\sigma}(D)\oplus\{\nabla p\in L^q(D)\mid 
p\in L^q_{\rm{loc}}(\overline{D})\},
\end{align*}
see Fujiwara and Morimoto \cite{fumo1977}, 
Miyakawa \cite{miyakawa1982}, and Simader and Sohr \cite{siso1992}. 
Let $P_q$ denote the Fujita-Kato projection 
from $L^q(D)$ onto $L^q_{\sigma}(D)$ 
associated with the decomposition. We remark that the adjoint operator 
of $P_q$ coincides 
with $P_{q/(q-1)}$. We simply write $P=P_q$\,. By real interpolation, 
it is possible to extend $P$ to a bounded operator on $L^{q,r}(D)$. 
We then define the solenoidal Lorentz spaces 
$L^{q,r}_{\sigma}(D)$ by
\begin{align*}
L^{q,r}_{\sigma}(D)=PL^{q,r}(D)=
\left(L^{q_0}_\sigma(D),L^{q_1}_\sigma(D)\right)_{\theta,r},
\end{align*}
where $1<q_0<q<q_1<\infty$ and $0<\theta<1$ 
satisfy $1/q=(1-\theta)/q_0+\theta/q_1$, 
while $1\leq r\leq \infty$, see Borchers and Miyakawa \cite{bomi1995}. 
We then have the duality relation 
$L^{q,r}_{\sigma}(D)^*=L_{\sigma}^{q/(q-1),r/(r-1)}(D)$ 
for $1<q<\infty$ and $1\leq r<\infty$. 
We denote various constants by $C$ 
and they may change from line to line. 
The constant dependent on $A,B,\cdots$ is denoted by $C(A,B,\cdots)$. 
Finally, if there is no confusion, 
we use the same symbols for denoting spaces of scalar-valued functions and 
those of vector-valued ones.   

\subsection{Main theorems}\label{maintheorem2}
~~~It is reasonable to look for a solution to (\ref{NS2}) of the form 
\begin{align*}
u(x,t)=v(x,t)+\psi(t)u_s,\quad p(x,t)=\phi(x,t)+\psi(t)p_s.
\end{align*}
Then the perturbation $(v,\phi)$ 
satisfies the following initial boundary value problem
\begin{align}
\left\{
\begin{array}{r@{}c@{}ll}
\partial_t v&{}={}&\Delta v+(\psi(t)\omega_0\times x)\cdot\nabla v
-\psi(t)\omega_0\times v\\
&&\hspace{4.3cm}+(Gv)(x,t)+H(x,t)-\nabla \phi,
&x\in D,t>0,\\
\nabla\cdot v&{}={}&0,\hspace{0.3cm} x\in D,t\geq 0,\\
v|_{\partial D}&{}={}&0,\hspace{0.3cm}t> 0,\label{a}\\
v&{}\rightarrow{}&0\hspace{0.47cm}{\rm{as}}~|x|\rightarrow \infty,\\
v(x,0)&{}={}&0,\hspace{0.3cm} x\in D,
\end{array}\right.
\end{align}
where 
\begin{align}
&(Gv)(x,t)=-v\cdot\nabla v-\psi(t)v\cdot\nabla u_s-\psi(t)u_s\cdot \nabla v,
\label{g}\\
&H(x,t)=\psi(t)(\psi(t)-1)\{-u_s\cdot \nabla u_s-\omega_0
\times u_s+(\omega_0\times x)\cdot\nabla u_s\}
-\psi'(t)u_s.\label{h}
\end{align}
In what follows, we concentrate ourselves 
on the problem (\ref{a}) instead of (\ref{NS2}). 
In fact, if we obtain the solution $v$ of (\ref{a}) 
which converges to $0$ as $t\rightarrow \infty$, 
the solution $u$ of (\ref{NS2}) converges to $u_s$ as 
$t\rightarrow \infty$. 
By using the evolution operator $\{T(t,s)\}_{t\geq s\geq 0}$ 
on $L^q_{\sigma}(D)~(1<q<\infty$) associated with (\ref{EV1}), 
problem (\ref{a}) is converted into
\begin{align}\label{integraleq}
v(t)=&\int_0^t T(t,\tau)P[(Gv)(\tau)+H(\tau)]\,d\tau.
\end{align}
\par We are now in a position to give our attainability theorem.
\begin{thm}\label{thm1}
Let $\psi$ be a function on $\R$ satisfying (\ref{psidef}) and 
put $\alpha:=\displaystyle\max_{t\in\R}|\psi'(t)|.$ 
For $q\in (6,\infty)$, there exists a constant $\delta(q)>0$ 
such that if $(\alpha+1)|a|\leq \delta,$ 
problem (\ref{integraleq}) admits a 
solution $v$ which possesses the following properties:
\begin{align*}
({\rm{i}})\,
v\in BC_{w^*}\big((0,\infty);L^{3,\infty}_{\sigma}(D)\big),~\,
\|v(t)\|_{3,\infty}\rightarrow 0~\, {\rm{as}}~ t\rightarrow 0,~\,
\sup_{0<t<\infty}\|v(t)\|_{3,\infty}\leq C(\alpha+1)|a|,
\end{align*}
where $BC_{w^*}(I;X)$  
is the set of bounded and weak-$\ast$ continuous  
functions on the interval $I$ with values in $X$, the constant $C$  
is independent of $a$ and $\psi$;
\begin{align*}
&({\rm{ii}})\,v\in C\big((0,\infty);L^r_{\sigma}(D)\big)
\cap C_{w^*}\big((0,\infty);L^\infty(D)\big),~ 
\nabla v\in C_w\big((0,\infty);L^r(D)\big)~{\rm for}~r\in (3,\infty);\\
&({\rm{iii}})\,({\rm{Decay}})\quad 
\|v(t)\|_{r}=O(t^{-\frac{1}{2}+\frac{3}{2r}})\quad 
{\rm{as}}~t\rightarrow \infty
\quad {\rm{for~all~}} r\in(3,q),\\
&\hspace{2.585cm}\|v(t)\|_{q,\infty}
=O(t^{-\frac{1}{2}+\frac{3}{2q}})\quad 
{\rm{as}}~t\rightarrow \infty,\\
&\hspace{2.585cm}\|v(t)\|_{r}=O(t^{-\frac{1}{2}+\frac{3}{2q}})
\quad {\rm{as}}~t\rightarrow \infty
\quad {\rm{for~all~}} r\in(q,\infty].
\end{align*}
\end{thm}

\begin{rmk}
We can obtain the $L^q$ decay of $v(t)$ like 
$O(t^{-1/2+3/(2q)}\log t)$ as $t\rightarrow\infty$, 
but it is not clear whether $\|v(t)\|_q=O(t^{-1/2+3/(2q)})$ holds. 
\end{rmk}

To prove Theorem \ref{thm1}, 
the key step is to construct a solution of the weak formulation 
\begin{align}\label{NS6}
(v(t),\varphi)=\int_0^t
\big(v(\tau)\otimes v(\tau)
&+\psi(\tau)\{v(\tau)\otimes u_s+u_s\otimes v(\tau)\},
\nabla T(t,\tau)^*\varphi\big)\,d\tau\nonumber\\&
+\int_0^t(H(\tau),T(t,\tau)^*\varphi)\,d\tau,
\quad \forall \varphi\in C^{\infty}_{0,\sigma}(D)
\end{align} 
as in Yamazaki \cite{yamazaki2000}, where $T(t,\tau)^*$ denotes the 
adjoint of 
$T(t,\tau)$ and, here and in what follows, 
$(\cdot,\cdot)$ stands for various duality pairings. 
In this paper, a function $v$ is called a solution of (\ref{NS6}) if 
$v\in L^{\infty}_{\rm{loc}}\big([0,\infty);L^{3,\infty}_{\sigma}(D)\big)$ 
satisfies (\ref{NS6}) for a.e. $t$. 
By following Yamazaki's approach, we can easily see that 
the solution obtained in Theorem \ref{thm1} is unique in the small, see 
Proposition \ref{thm2}. In the following theorem, we give another result 
on the uniqueness without assuming the smallness of solutions. 
\begin{thm}\label{thm3}
Let $q\in (3,\infty).$ 
Then there exists a constant $\widetilde{\delta}>0$ 
independent of $q$ and $\psi$ 
such that if $|a|\leq \widetilde{\delta}$, 
problem (\ref{NS6}) admits at most one solution 
within the class
\begin{align*}
\big\{v\in L^{\infty}_{\rm{loc}}
\big([0,\infty);L^{3,\infty}_{\sigma}(D)\big)\cap 
L^{\infty}_{\rm{loc}}
\big(0,\infty;L^q_\sigma(D)\big)\,\big|\,
\lim_{t\rightarrow 0}\|v(t)\|_{3,\infty}=0\big\}.
\end{align*}
\end{thm}

\begin{rmk}
Theorem \ref{thm3} asserts that if the angular velocity is small enough 
and if $\widetilde{v}$ is a solution 
within the class above which is not necessarily small, 
then it coincides with the solution obtained in Theorem \ref{thm1}. 
\end{rmk}

\section{Preliminary results}\label{preliminary}
~~~~In this section, we prepare some results on the 
stationary solutions and the evolution operator. 
For the stationary problem (\ref{sta}), Galdi \cite{galdi2003} proved the following result. 
\begin{prop}[\cite{galdi2003}]\label{galdi2003}
There exists a constant $\eta\in(0,1]$ such that if 
$|\omega_0|=|a|\leq \eta$, the stationary problem (\ref{sta}) admits a unique solution $(u_s,p_s)$ with 
the estimate
\begin{align*}
\sup_{x\in D}\big\{(1+|x|)|u_s(x)|\big\}+\sup_{x\in D}\big\{(1+|x|^2)(|\nabla u_s(x)|+|p_s(x)|)\big\}\leq C|a|,
\end{align*}
where the constant $C$ is independent of $a$.
\end{prop}
From now on, we assume that the angular velocity $\omega_0=(0,0,a)^\top$ always satisfies 
$|\omega_0|=|a|\leq \eta$.
Proposition \ref{galdi2003} then yields
\begin{align*}
u_s\in L^{3,\infty}(D)\cap L^{\infty}(D),\quad \nabla u_s\in L^{\frac{3}{2},\infty}(D)\cap L^{\infty}(D),
\quad |x|\nabla u_s\in L^{3,\infty}(D)\cap L^\infty(D)
\end{align*}
and 
\begin{align}\label{hstaest2}
H(t)\in L^{3,\infty}(D),\quad 
\|H(t)\|_{3,\infty}\leq C(a^2+\alpha|a|)
\end{align}
for all $t>0$. Here, $H(t)$ is defined by (\ref{h}) and 
$\alpha=\displaystyle \sup_{t\in\R}|\psi'(t)|.$

\par We next collect some results on the evolution operator associated 
with (\ref{EV1}). 
We define the linear operator by 
\begin{align*}
\D_q(L(t))&=\{
u\in L^q_{\sigma}(D)\cap W_0^{1,q}(D)\cap W^{2,q}(D)\mid
(\omega_0\times x)\cdot \nabla u\in L^q(D)\},\\
L(t)u&=-P[\Delta u+ (\psi(t)\omega_0\times x)\cdot \nabla u- \psi(t)\omega_0 \times u].
\end{align*}
Then the problem (\ref{EV1}) is formulated as 
\begin{align}\label{EV2}
\partial_t u+L(t)u=0,\quad t\in (s,\infty);\quad u(s)=f 
\end{align}
in $L^q_{\sigma}(D)$. 
We can see that (\ref{psidef}) implies
\begin{align}\label{globalholder}
\psi(t)\omega_0\in C^\theta ([0,\infty);\R^3)\cap L^{\infty}(0,\infty;\R^3) 
\end{align}
for all $\theta\in (0,1)$. In fact, we have 
\begin{align}\label{supholder}
\sup_{0\leq t<\infty}|\psi(t)\omega_0|=|a|,\quad 
\sup_{0\leq s<t<\infty}\frac{|\psi(t)\omega_0-\psi(s)\omega_0|}{(t-s)^{\theta}}\leq |a|
\max_{t\in \R}|\psi'(t)|
\end{align}
for all $\theta\in (0,1)$. 
We fix, for instance, $\theta=1/2.$ 
Under merely the local H\"{o}lder continuity of the angular velocity, 
Hansel and Rhandi \cite{harh2014} proved the following proposition 
(see also Hishida \cite{hishidapre} concerning the assertion 1). 
Indeed they did not derive the assertion 4, 
but it directly follows from the real interpolation.
For completeness, we give its proof. 
\begin{prop}[\cite{harh2014}]\label{harhprop}
Let $1<q<\infty$. Suppose (\ref{psidef}). 
The operator family $\{L(t)\}_{t\geq 0}$ generates 
a strongly continuous evolution operator 
$\{T(t,s)\}_{t\geq s\geq 0}$ on $L^q_{\sigma}(D)$ with the following properties:
\begin{enumerate}
\item Let $q\in(3/2,\infty)$ and $s\geq 0$. 
For every $f\in Z_q(D)$ and $t\in (s,\infty),$ we have $T(t,s)f\in Y_q(D)$ and 
$T(\cdot,s)f\in C^1((s,\infty); L^q_{\sigma}(D))$ with
\begin{align*}
\partial_t T(t,s)f+L(t)T(t,s)f=0,\quad t\in(s,\infty) 
\end{align*}
in $L^q_{\sigma}(D)$, where 
\begin{align*}
&Y_q(D)=\{u\in L^q_{\sigma}(D)\cap W_0^{1,q}(D)
\cap W^{2,q}(D)\mid|x|\nabla u\in L^q(D)\},\\
&Z_q(D)=\{u\in L^q_\sigma(D)\cap W^{1,q}(D)\mid|x|\nabla u\in L^q(D)\}.
\end{align*}
\item For every $f\in Y_q(D)$ and $t>0$, 
we have $T(t,\cdot)f\in C^1\big([0,t]; L^q_{\sigma}(D)\big)$ with
\begin{align*}
\partial_s T(t,s)f=T(t,s)L(s)f\quad s\in [0,t]
\end{align*}
in $L^q_{\sigma}(D)$.
\item Let $1<q\leq r<\infty,$ and 
$m,\T\in (0,\infty)$. There is a constant $C=C(q,r,m,\T,D)$ such that 
\begin{align}\label{gradT}
\|\nabla T(t,s)f\|_{r}
\leq C(t-s)^{-\frac{3}{2}(\frac{1}{q}-\frac{1}{r})-\frac{1}{2}}
\|f\|_{q}
\end{align}
holds for all $0\leq s<t\leq \T$ and $f\in L^{q}_{\sigma}(D)$ whenever 
\begin{align}\label{condition m}
(1+\max_{t\in \R}|\psi'(t)|\,)|a|\leq m
\end{align}
is satisfied. 
\item Let $1<q<r<\infty, 1\leq\rho_1,\rho_2\leq\infty
~{\rm{and}}~ m,\T\in (0,\infty)$. 
There is a constant $C=C(q,r,\rho_1,\rho_2,m,\T,D)$ such that 
\begin{align}\label{gradT2}
\|\nabla T(t,s)f\|_{r,\rho_2}
\leq C(t-s)^{-\frac{3}{2}(\frac{1}{q}-\frac{1}{r})-\frac{1}{2}}
\|f\|_{q,\rho_1}
\end{align}
holds for all $0\leq s<t\leq \T$ and $f\in L^{q,\rho_1}_{\sigma}(D)$ 
whenever (\ref{condition m}) is satisfied. 
\end{enumerate}
\end{prop}
\hspace{-0.6cm}{\bf{Proof of the assertion 4.}}~
We choose $r_0,r_1$ such that $1<q<r_0<r<r_1<\infty.$ 
From the assertion 3 and the real interpolation, we have 
\begin{align}
&\|\nabla T(t,s)f\|_{r_0,\rho_1}\leq C(t-s)^{-\frac{3}{2}(\frac{1}{q}-\frac{1}{r_0})-\frac{1}{2}}
\|f\|_{q,\rho_1},\label{lqlr00}\\
&\|\nabla T(t,s)f\|_{r_1,\rho_1}\leq C(t-s)^{-\frac{3}{2}(\frac{1}{q}-\frac{1}{r_1})-\frac{1}{2}}
\|f\|_{q,\rho_1}.\label{lqlr1}
\end{align}
By the reiteration theorem 
for real interpolation 
(see for instance \cite[Theorem 3.5.3]{belobook1976}), we obtain 
\begin{align}\label{realinterpolation}
L_\sigma^{r,\rho_2}(D)=
(L_{\sigma}^{r_0,\rho_1}(D),L_{\sigma}^{r_1,\rho_1}(D))_{\beta,\rho_2},\quad 
\|u\|_{r,\rho_2}\leq C\|u\|_{r_0,\rho_1}^{1-\beta}\|u\|^{\beta}_{r_1,\rho_1},\quad 
\frac{1}{r}=\frac{1-\beta}{r_0}+\frac{\beta}{r_1}
\end{align}
which combined with (\ref{lqlr00}) and 
(\ref{lqlr1}) concludes (\ref{gradT2}).\qed
\vspace{0.2cm}
\par We know that the adjoint operator $T(t,s)^*$ is also a strongly 
continuous evolution operator and 
satisfies the backward semigroup property
\begin{align*}
T(\tau,s)^*T(t,\tau)^*=T(t,s)^*\quad (t\geq \tau\geq s\geq 0),\quad T(t,t)^*=I,
\end{align*} 
see Hishida \cite[Subsection 2.3]{hishida2018}. 
Under the assumption (\ref{globalholder}) with some $\theta\in(0,1)$, 
Hishida \cite{hishida2018,hishidapre} established 
the following 
$L^q$-$L^r$ decay estimates. The assertion 3 is not found there 
but can be proved in the same way as above. 
We note that the idea of deduction of (\ref{intesti}) 
below is due to Yamazaki \cite{yamazaki2000} once we have the assertion 5.
\begin{prop}[\cite{hishida2018,hishidapre}]\label{hishidaprop}
\quad Let $m\in (0,\infty)$ and suppose (\ref{psidef}).
\begin{enumerate}
\item Let $1<q\leq r\leq\infty~(q\ne\infty)$. 
Then there exists a constant $C=C(m,q,r,D)$ such that
\begin{align}
&\|T(t,s)f\|_{r}\leq
C(t-s)^{-\frac{3}{2}(\frac{1}{q}-\frac{1}{r})}\|f\|_{q}\label{lqlr0}\\
&\|T(t,s)^*g\|_{r}\leq
C(t-s)^{-\frac{3}{2}(\frac{1}{q}-\frac{1}{r})}\|g\|_{q}\label{ad0}
\end{align}
hold for all $t> s\geq 0$ and $f,g\in L^q_{\sigma}(D)$ 
whenever (\ref{condition m}) is satisfied.
\item Let $1<q\leq r<\infty, ~1\leq \rho\leq \infty$. 
Then there exists a constant $C=C(m,q,r,\rho,D)$ 
such that
\begin{align}
&\|T(t,s)f\|_{r,\rho}\leq
C(t-s)^{-\frac{3}{2}(\frac{1}{q}-\frac{1}{r})}
\|f\|_{q,\rho}\label{lqlr}\\
&\|T(t,s)^*g\|_{r,\rho}
\leq 
C(t-s)^{-\frac{3}{2}(\frac{1}{q}-\frac{1}{r})}\|g\|_{q,\rho}\label{ad}
\end{align}
hold for all $t> s\geq 0$ and $f,g\in L^{q,\rho}_{\sigma}(D)$ 
whenever (\ref{condition m}) is satisfied.
\item Let $1<q<r<\infty, 1\leq \rho_1,\rho_2\leq \infty$. 
Then there exists a constant 
$C=C(m,q,r,\rho_1,\rho_2,D)$ such that
\begin{align}
&\|T(t,s)f\|_{r,\rho_2}\leq C(t-s)^{-\frac{3}{2}(\frac{1}{q}-\frac{1}{r})}\|f\|_{q,\rho_1}\label{lqlr2}\\
&\|T(t,s)^*g\|_{r,\rho_2}\leq C(t-s)^{-\frac{3}{2}(\frac{1}{q}-\frac{1}{r})}\|g\|_{q,\rho_1}\label{ad2}
\end{align}
hold for all $t> s\geq 0$ and $f,g\in L^{q,\rho_1}_{\sigma}(D)$ 
whenever (\ref{condition m}) is satisfied.
\item Let $1<q\leq r\leq 3$.  
Then there exists a constant $C=C(m,q,r,D)$ 
such that
\begin{align}
&\|\nabla T(t,s)f\|_{r}\leq C(t-s)^{-\frac{3}{2}(\frac{1}{q}-\frac{1}{r})-\frac{1}{2}}
\|f\|_{q}\label{grad0}\\
&\|\nabla T(t,s)^*g\|_{r}\leq C(t-s)^{-\frac{3}{2}(\frac{1}{q}-\frac{1}{r})-\frac{1}{2}}
\|g\|_{q}\label{adgrad0}
\end{align}
hold for all $t> s\geq 0$ and $f,g\in L^{q}_{\sigma}(D)$ 
whenever (\ref{condition m}) is satisfied.
\item Let $1<q\leq r\leq 3,~1\leq \rho<\infty$.  
Then there exists a constant $C=C(m,q,r,\rho,D)$ such that
\begin{align}
&\|\nabla T(t,s)f\|_{r,\rho}\leq C(t-s)^{-\frac{3}{2}(\frac{1}{q}-\frac{1}{r})-\frac{1}{2}}
\|f\|_{q,\rho}\label{grad}\\
&\|\nabla T(t,s)^*g\|_{r,\rho}\leq C(t-s)^{-\frac{3}{2}(\frac{1}{q}-\frac{1}{r})-\frac{1}{2}}
\|g\|_{q,\rho}\label{adgrad}
\end{align}
hold for all $t> s\geq 0$ and $f,g\in L^{q,\rho}_{\sigma}(D)$ 
whenever (\ref{condition m}) is satisfied.
\item Let $1<q\leq r\leq 3$ with $1/q-1/r=1/3$.  
Then there exists a constant $C=C(m,q,D)$ such that
\begin{align}\label{intesti} 
\int_0^t\|\nabla T(t,s)^*g\|_{r,1}\,ds\leq C\|g\|_{q,1}
\end{align}
holds for all $t>0$ and $g\in L^{q,1}_{\sigma}(D)$ 
whenever (\ref{condition m}) is satisfied.
\end{enumerate}
\end{prop}

To prove the $L^{\infty}$ decay estimate in Theorem \ref{thm1}, 
we also prepare the following $L^1$-$L^r$ estimates. 
The following estimates for data being in $C_0^\infty(D)^3$ are enough 
for later use, but it is clear that, for instance, 
the composite operator $T(t,s)P$ extends to 
a bounded operator from $L^1(D)$ to $L^r_\sigma(D)$ 
with the same estimate.

\begin{lem}\label{lem,infty,1}
Let $m\in (0,\infty)$ and suppose (\ref{psidef}).
\begin{enumerate}
\item Let $1<r<\infty$. 
Then there is a constant $C=C(m,r,D)>0$ such that
\begin{align}
&\|T(t,s)Pf\|_r\leq C(t-s)^{-\frac{3}{2}(1-\frac{1}{r})}\|f\|_1,\label{l1lr}\\
&\|T(t,s)^*Pg\|_r\leq C(t-s)^{-\frac{3}{2}(1-\frac{1}{r})}\|g\|_1\label{l1lrad}
\end{align}
for all $t>s\geq 0$ and $f,g\in C^\infty_0(D)^3$ whenever 
(\ref{condition m}) is satisfied.
\item Let $1<r<\infty$ and $1\leq\rho\leq \infty$. 
Then there is a constant 
$C=C(m,r,\rho,D)>0$ such that
\begin{align}
&\|T(t,s)Pf\|_{r,\rho}\leq C(t-s)^{-\frac{3}{2}(1-\frac{1}{r})}\|f\|_1,\label{l1lr2}\\
&\|T(t,s)^*Pg\|_{r,\rho}\leq C(t-s)^{-\frac{3}{2}(1-\frac{1}{r})}\|g\|_1\label{l1lr2ad}
\end{align}
for all $t>s\geq 0$ and $f,g\in C^\infty_0(D)^3$ 
whenever (\ref{condition m}) is satisfied.
\item Let $1<r\leq 3,\,1\leq\rho<\infty$.  
Then there is a constant 
$C=C(m,r,\rho,D)>0$ such that
\begin{align}
\|\nabla T(t,s)Pf\|_{r,\rho}\leq C(t-s)^{-\frac{3}{2}(1-\frac{1}{r})-\frac{1}{2}}\|f\|_1,
\label{gradl1lr}\\
\|\nabla T(t,s)^*Pg\|_{r,\rho}\leq C(t-s)^{-\frac{3}{2}(1-\frac{1}{r})-\frac{1}{2}}\|g\|_1\label{gradl1lrad}
\end{align}
for all $t>s\geq 0$ and $f,g\in C^{\infty}_0(D)^3$ whenever 
(\ref{condition m}) is satisfied.
\end{enumerate}
\end{lem}
\begin{proof}
The proof is simply based on duality argument 
(see Koba \cite[Lemma 2.15]{koba2017}), however, 
we give it for completeness. 
Let $1<q\leq r<\infty$ and $1/r+1/r'=1$. 
By using (\ref{ad0}), 
we see that 
\begin{align}\label{argue}
|(T(t,s)Pf,\varphi)|=|(f,T(t,s)^*\varphi)|\leq \|f\|_1\|T(t,s)^*\varphi\|_\infty
\leq C(t-s)^{-\frac{3}{2}(1-\frac{1}{r})}\|f\|_1\|\varphi\|_{r'}
\end{align}
for all $\varphi \in L^{r'}_\sigma(D)$, 
which implies (\ref{l1lr}). We next show (\ref{l1lr2}). 
We fix $q$ such that $1<q<r$. 
Combining the estimate (\ref{lqlr2}) with (\ref{l1lr}), 
we have 
\begin{align*}
\|T(t,s)Pf\|_{r,\rho}\leq 
C(t-s)^{-\frac{3}{2}(\frac{1}{q}-\frac{1}{r})}
\left\|T\left(\frac{t+s}{2},s\right)Pf\right\|_q
\leq C(t-s)^{-\frac{3}{2}(1-\frac{1}{r})}\|f\|_1.
\end{align*}
Finally, in view of (\ref{grad}) and (\ref{l1lr2}), we have 
\begin{align*}
\|\nabla T(t,s)Pf\|_{r,\rho}\leq 
C(t-s)^{-\frac{1}{2}}
\left\|T\left(\frac{t+s}{2},s\right)Pf\right\|_{r,\rho}
\leq C(t-s)^{-\frac{3}{2}(1-\frac{1}{r})-\frac{1}{2}}\|f\|_1
\end{align*} 
which implies (\ref{gradl1lr}). 
The proof for the adjoint $T(t,s)^*$ is accomplished in the same way.
\end{proof}

\section{Proof of the main theorems}\label{proof}
~~~In this section we prove the main theorems 
(Theorem~\ref{thm1} and Theorem~\ref{thm3}). 
We first give some key estimates and then show Theorem~\ref{thm3}. 
After that, following Yamazaki \cite{yamazaki2000}, 
we construct a solution with some decay properties 
for (\ref{NS6}) 
and then derive the $L^\infty$ decay of the solution. 
We finally identify 
the solution above with a local solution 
possessing better regularity 
for the integral equation (\ref{integraleq}) in a 
neighborhood of each time $t>0$. 
\par Let us define the function spaces
\begin{align*}
&X=\big\{v\in BC_{w^*}\big((0,\infty);L^{3,\infty}_{\sigma}(D)\big)
\,\big|\,\lim_{t\rightarrow 0}
\|v(t)\|_{3,\infty}=0\big\},\\
&X_q=\big\{v\in X\,\big|\,t^{\frac{1}{2}-\frac{3}{2q}}v
\in BC_{w^*}\big((0,\infty);L_{\sigma}^{q,\infty}(D)
\big)\big\},~~3<q<\infty.
\end{align*}
Both are Banach spaces endowed with 
norms $\|\cdot\|_X=\|\cdot\|_{X,\infty}$ and 
$\|\cdot\|_{X_q}=\|\cdot\|_{X_q,\infty}$, respectively, where 
\begin{align*}
&\|v\|_{X,t}=\sup_{0<\tau<t}\|v(\tau)\|_{3,\infty},\quad 
\|v\|_{X_q,t}=\|v\|_{X,t}
+[v]_{q,t},\quad 
[v]_{q,t}=\sup_{0<\tau<t}
\tau^{\frac{1}{2}-\frac{3}{2q}}\|v(\tau)\|_{q,\infty}
\end{align*}
for $t\in (0,\infty].$ 
\begin{lem}\label{key}
\begin{enumerate}
\item ~Let $v,w\in X$ and set
\begin{align*}
&\braket{\I(v,w)(t),\varphi}:=
\int_0^t (v(\tau)\otimes w(\tau),
\nabla T(t,\tau)^*\varphi)\,d\tau,\\
&\braket{\J(v)(t),\varphi}:=
\int_0^t (\psi(\tau)\{v(\tau)\otimes u_s+u_s\otimes v(\tau)\},
\nabla T(t,\tau)^*\varphi)\,d\tau
\end{align*}
for all $\varphi\in C_{0,\sigma}^{\infty}(D)$. 
Then $\I(v,w),\J(v)\in X$ and there exists a positive 
constant $C$ such that
\begin{align}\label{e1}
\|\I(v,w)\|_{X,t}\leq C\|v\|_{X,t}\|w\|_{X,t},\quad\quad 
\|\J(v)\|_{X,t}\leq C\|u_s\|_{3,\infty}\|v\|_{X,t}
\end{align}
hold for any $v,w\in X$ and $t\in(0,\infty]$.
\item ~Let $q\in(3,\infty)$. If $v\in X_q$, $w\in X$, 
then $\I(v,w),\J(v)\in X_q$ and 
there exists a positive constant $C=C(q)$ such that
\begin{align}\label{e2}
\|\I(v,w)\|_{X_q,t}\leq C\|v\|_{X_q,t}\|w\|_{X,t},
\quad \|\J(v)\|_{X_q,t}\leq C\|u_s\|_{3,\infty}\|v\|_{X_q,t}
\end{align}
hold for every $v\in X_q$, $w\in X$ and $t\in(0,\infty]$.
\item ~We set 
\begin{align*}
\braket{\K(t),\varphi}:=\int_0^t (H(\tau),T(t,\tau)^*\varphi)\,d\tau
\end{align*}
for $\varphi\in C_{0,\sigma}^\infty(D)$. 
Let $q\in(3,\infty)$. Then $\K\in X_q$ and 
there exist positive constants $C$ independent of $q$ 
and $C'=C'(q)$ such that
\begin{align}\label{kest}
\|\K\|_{X,t}\leq C(a^2+\alpha|a|),\quad 
\|\K\|_{X_q,t}\leq C'(a^2+\alpha|a|)
\end{align}
hold for every $t\in (0,\infty]$.
\end{enumerate}
\end{lem}
\begin{proof}
Estimates (\ref{e1}) and (\ref{e2}) can be proved in the same way as done by 
Yamazaki \cite[Lemma 6.1.]{yamazaki2000}, see also 
Hishida and Shibata \cite[Section 8]{hish2009}, however, we 
briefly give the proof of (\ref{e1})$_1$ and (\ref{e2})$_1$. 
By (\ref{intesti}), we have
\begin{align*}
|\braket{\I(v,w)(t),\varphi}|\leq \|v\|_{X,t}\|w\|_{X,t}
\int_0^t\|\nabla T(t,\tau)^*\varphi\|_{3,1}
\leq C\|v\|_{X,t}\|w\|_{X,t}\|\varphi\|_{\frac{3}{2},1},
\end{align*}
which yields (\ref{e1})$_1$. 
We choose $r$ such that $1/3+1/q+1/r=1$ to find 
\begin{align*}
|\braket{\I(v,w)(t),\varphi}|\leq [v]_{X_q,t}\|w\|_{X,t}
\int_0^t\tau^{-\frac{1}{2}+\frac{3}{2q}}
\|\nabla T(t,\tau)^*\varphi\|_{r,1}\,d\tau= [v]_{q,t}\|w\|_{X,t}
\left(\int_0^{\frac{t}{2}}+\int_{\frac{t}{2}}^t\right).
\end{align*}
In view of (\ref{adgrad}), we have 
\begin{align*}
\int_0^{\frac{t}{2}}\leq C\int_0^{\frac{t}{2}}\tau^{-\frac{1}{2}
+\frac{3}{2q}}(t-\tau)^{-1}\,d\tau\|\varphi\|_{q',1}
\leq Ct^{-\frac{1}{2}+\frac{3}{2q}}\|\varphi\|_{q',1},
\end{align*}
where $1/q+1/q'=1$, 
whreas, (\ref{intesti}) implies 
\begin{align*}
\int_{\frac{t}{2}}^t
\leq \left(\frac{t}{2}\right)^{-\frac{1}{2}+\frac{3}{2q}}
\int_0^t\|\nabla T(t,\tau)^*\varphi\|_{r,1}\,d\tau
\leq Ct^{-\frac{1}{2}+\frac{3}{2q}}\|\varphi\|_{q',1}
\end{align*}
from which together with (\ref{e1})$_1$, we obtain (\ref{e2})$_1$. 
The estimate (\ref{e1}) leads us to
\begin{align*}
\lim_{t\rightarrow 0}\|\I(v,w)(t)\|_{3,\infty}=0,\quad 
\lim_{t\rightarrow 0}\|\J(v)(t)\|_{3,\infty}=0.
\end{align*}
\par Let us consider the weak-$\ast$ continuity 
of $\I(v,w)$ with values in $L^{3,\infty}_{\sigma}(D)$ 
(resp. $L^{q,\infty}_{\sigma}(D)$) when $v\in X$ (resp. $v\in X_q$), 
$w\in X$. Here, we need a different argument from 
\cite{yamazaki2000} because of the non-autonomous character 
as well as the non-analyticity of the corresponding semigroup. 
Since $C_{0,\sigma}^{\infty}(D)$ is dense in 
$L_{\sigma}^{\kappa,1}$ ($\kappa=3/2,\,q'$) and since
we know (\ref{e1}) and (\ref{e2}), it suffices to show that
\begin{align}\label{w*contiI}
|\braket{\I(v,w)(t)-\I(v,w)(\sigma),\varphi}|
\rightarrow 0~~ {\rm{as}} ~~\sigma\rightarrow t 
\end{align}
for all $0<t<\infty$ and $\varphi\in C^{\infty}_{0,\sigma}(D)$. 
Let $0<\sigma<t$. By using the backward semigroup property, we have
\begin{align*}
|\braket{\I(v,w)(t)-\I(v,w)(\sigma),\varphi}|&\leq 
\int_0^\sigma|(v(\tau)\otimes w(\tau),
\nabla T(\sigma,\tau)^*(T(t,\sigma)^*\varphi-\varphi)|\,d\tau
\\&\hspace{1.5cm}
+\int_\sigma^t
|(v(\tau)\otimes w(\tau),\nabla T(t,\tau)^*\varphi)|\,d\tau
=:I_1+I_2.
\end{align*}
The estimate (\ref{intesti}) yields
\begin{align*}
I_1&\leq \|v\|_{X,t}\|w\|_{X,t}\int_0^\sigma\|
\nabla T(\sigma,\tau)^*(T(t,\sigma)^*\varphi-\varphi)\|_{3,1}\,d\tau\\
&\leq C\|v\|_{X,t}\|w\|_{X,t}\,
\|T(t,\sigma)^*\varphi-\varphi\|_{\frac{3}{2},1}\rightarrow 0 
\quad{\rm{as}}~ \sigma \rightarrow t.
\end{align*}
Furthermore, (\ref{adgrad}) yields
\begin{align*}
I_2\leq \|v\|_{X,t}\|w\|_{X,t}\int_\sigma^t
\|\nabla T(t,\tau)^*\varphi\|_{3,1}\,d\tau
\leq C\|v\|_{X,t}\|w\|_{X,t}(t-\sigma)^{\frac{1}{2}}\|\varphi\|_{3,1}
\rightarrow 0 \quad {\rm{as}}~\sigma \rightarrow t.
\end{align*} 
We can discuss the other case $0<t<\sigma$ similarly and 
thus we obtain (\ref{w*contiI}). 
By the same manner, we can obtain the desired weak-$\ast$ continuity 
of $\J$. We thus conclude the assertion 1 and 2. 
\par We next consider $\K(t)$. 
We use (\ref{ad}) as well as (\ref{hstaest2})
to obtain 
\begin{align*}
|\braket{\K(t),\varphi}|\leq C(a^2+\alpha|a|)
\int_0^{\min\{1,t\}}|\|T(t,\tau)^*\varphi\|_{\frac{3}{2},1}\,d\tau
\leq C(a^2+\alpha|a|)\min\{1,t\}\|\varphi\|_{\frac{3}{2},1}
\end{align*}
for $\varphi\in C_{0,\sigma}^\infty(D)$ and $t>0$ 
which yields $\K(t)\in L^{3,\infty}_\sigma(D)$ with 
\begin{align*}
\|\K\|_{X,t}\leq C(a^2+\alpha|a|)
\quad {\rm{for}} ~t\in(0,\infty],\quad   
\lim_{t\rightarrow 0} \|\K(t)\|_{3,\infty}=0.
\end{align*}
To derive the estimate 
$[\K]_{q,t}\leq C(a^2+\alpha|a|)$, 
we consider two cases: $0<t\leq 2$ and $t\geq 2$. 
For $0<t\leq2$, (\ref{ad}) yields 
\begin{align*}
|\braket{\K(t),\varphi}|
\leq C(a^2+\alpha|a|)
\int_0^t\|T(t,\tau)^*\varphi\|_{\frac{3}{2},1}\,d\tau
\leq C(a^2+\alpha|a|)\|\varphi\|_{q',1}.
\end{align*}
For $t\geq 2$, we have 
\begin{align*}
|\braket{\K(t),\varphi}|
\leq C(a^2+\alpha|a|)
\int_0^1\|T(t,\tau)^*\varphi\|_{\frac{3}{2},1}\,d\tau
\leq C(a^2+\alpha|a|)\,
t^{-\frac{1}{2}+\frac{3}{2q}}\|\varphi\|_{q',1}.
\end{align*}
We thus obtain (\ref{kest}). 
It remains to show the weak-$\ast$ continuity. To this end, it is 
sufficient to show that 
\begin{align}\label{w*contiJ} 
|\braket{\K(t)-\K(\sigma),\varphi}|\rightarrow 0~~ {\rm{as}} 
~\sigma\rightarrow t
\end{align}
for all $t\in (0,\infty)$ 
due to (\ref{kest}). To prove (\ref{w*contiJ}), 
we suppose  $0<\sigma<t$. 
We use the backward semigroup property to observe
\begin{align*}
\braket{\K(t)-\K(\sigma),\varphi}
=\int_0^\sigma\big(H(\tau),T(\sigma,\tau)^*(T(t,\sigma)^*
\varphi-\varphi)\big)\,d\tau
+\int_\sigma^t (H(\tau),T(t,\tau)^*\varphi)\,d\tau.
\end{align*}
By applying (\ref{ad}), we find that 
\begin{align*}
&\left|
\int_0^\sigma\big(H(\tau),T(\sigma,\tau)^*(T(t,\sigma)^*\varphi-\varphi)
\big)\,d\tau
\right|
\leq C(a^2+\alpha|a|)\,
\sigma\,\|T(t,\sigma)^*\varphi-\varphi\|_{\frac{3}{2},1}
\rightarrow 0
\quad {\rm{as}} ~\sigma\rightarrow t,\\
&\left|
\int_\sigma^t (H(\tau),T(t,\tau)^*\varphi)\,d\tau
\right|
\leq C(a^2+\alpha|a|)(t-\sigma)\|\varphi\|_{\frac{3}{2},1}
\rightarrow 0\quad {\rm{as}} ~\sigma\rightarrow t.
\end{align*}
The other case $t<\sigma$ is discussed similarly. Hence we have 
(\ref{w*contiJ}). The proof is complete.
\end{proof}

\hspace{-0.6cm}{\bf{Proof of Theorem \ref{thm3}.}}~
The idea of the proof is traced back to Fujita and Kato 
\cite[Theorem 3.1.]{fuka1964}. 
Let $v_1$ and $v_2$ be the solutions of (\ref{NS6}). Then we have 
\begin{align}
(v_1(t)-v_2(t),\varphi)=&\int_0^t
\big(v_1(\tau)\otimes(v_1(\tau)-v_2(\tau))+
(v_1(\tau)-v_2(\tau))
\otimes v_2(\tau)\nonumber\\&+\psi(\tau)(v_1(\tau)-v_2(\tau))\otimes u_s+
\psi(\tau)u_s\otimes(v_1(\tau)-v_2(\tau)),
\nabla T(t,\tau)^*\varphi\big)\,d\tau\label{v1v2}
\end{align}
for $\varphi\in C_{0,\sigma}^\infty(D).$ 
By applying (\ref{e1}) to (\ref{v1v2}) and 
by Proposition \ref{galdi2003}, 
we have
\begin{align*}
\|v_1-v_2\|_{X,t}&
\leq C(\|v_1\|_{X,t}+\|v_2\|_{X,t}+\|u_s\|_{3,\infty})
\|v_1-v_2\|_{X,t}\\
&\leq C(\|v_1\|_{X,t}+\|v_2\|_{X,t}+|a|)\|v_1-v_2\|_{X,t}.
\end{align*} 
Suppose 
\begin{align*}
|a|<\frac{1}{2C}=:\widetilde{\delta}.
\end{align*} 
Since $\|v_j(t)\|_{3,\infty}\rightarrow 0$ as $t\rightarrow 0$ ($j=1,2$), 
one can choose $t_0>0$ such that 
$C(\|v_1\|_{X,t_0}+\|v_2\|_{X,t_0})<1/2$, 
which implies $v_1=v_2$ on $(0,t_0]$. 
Hence, (\ref{v1v2}) is written as  
\begin{align}\label{v1v22}
(v_1(t)-v_2(t),\varphi)=
&\int_{t_0}^t\big(v_1(\tau)\otimes(v_1(\tau)-v_2(\tau))+
(v_1(\tau)-v_2(\tau))
\otimes v_2(\tau)\nonumber\\&
+\psi(\tau)(v_1(\tau)-v_2(\tau))\otimes u_s+
\psi(\tau)u_s\otimes(v_1(\tau)-v_2(\tau)),
\nabla T(t,\tau)^*\varphi\big)\,d\tau.
\end{align}
We fix $\T\in (t_0,\infty)$ and set 
$[v]_{q,t_0,t}=\displaystyle\sup_{t_0\leq\tau\leq t}\|v(\tau)\|_q$ 
for $t\in (t_0,\T)$. 
It follows from 
(\ref{v1v22}) that 
\begin{align}\label{v1v2r}
[v_1-v_2]_{q,t_0,t}\leq C_*(t-t_0)^{\frac{1}{2}-\frac{3}{2q}}
[v_1-v_2]_{q,t_0,t}\,,\quad 
t\in (t_0,\T),
\end{align}
where 
$C_*=C_*(t_0,\T)=C([v_1]_{q,t_0,\T}+[v_2]_{q,t_0,\T}+2\|u_s\|_q\big).$ 
In fact, the estimate (\ref{adgrad0}) yields
\begin{align*}
\int_{t_0}^t\big|&\big(v_1(\tau)\otimes(v_1(\tau)-v_2(\tau)),
\nabla T(t,\tau)^*\varphi\big)\big|\,d\tau\\
&\leq C[v_1(\tau)]_{q,t_0,\T}[v_1-v_2]_{q,t_0,t}
\int_{t_0}^t\|\nabla
T(t,\tau)^*\varphi\|_{(1-\frac{2}{q})^{-1}}\,d\tau\\
&\leq C[v_1(\tau)]_{q,t_0,\T}
[v_1-v_2]_{q,t_0,t}(t-t_0)^{\frac{1}{2}-\frac{3}{2q}}
\|\varphi\|_{(1-\frac{1}{q})^{-1}}
\end{align*}
for all $\varphi\in C_{0,\sigma}^{\infty}(D)$ and $t\in (t_0,\T)$.
Since the other terms in (\ref{v1v22}) are treated similarly, 
we obtain (\ref{v1v2r}). We take 
\begin{align*}
\xi=\min\left\{\left(\frac{1}{2C_*}\right)^
{(\frac{1}{2}-\frac{3}{2q})^{-1}},\T-t_0\right\}
\end{align*}
which leads us to $v_1=v_2$ on $(0,t_0+\xi).$ 
Even though we replace $t_0$ by $t_0+\xi,t_0+2\xi,\cdots$, 
we can discuss similarly. Hence, $v_1=v_2$ on $(0,\T)$. 
Since $\T$ is arbitrary, we conclude $v_1=v_2$. \qed
\vspace{0.2cm}
\par To prove Theorem \ref{thm1},
we begin to construct a solution of (\ref{NS6}) 
by applying Lemma \ref{key}.
\begin{prop}\label{thm2}
Let $\psi$ be a function on $\R$ satisfying (\ref{psidef}). 
We put $\alpha=\displaystyle\max_{t\in\R}|\psi'(t)|$. 
\begin{enumerate}
\item There exists $\delta_1>0$ such that 
if $(\alpha+1)|a| \leq \delta_1$, problem (\ref{NS6}) 
admits a unique solution 
within the class
\begin{align*}
\big\{v\in BC_{w^*}\big((0,\infty);L^{3,\infty}_{\sigma}(D)\big)\,\big|
\,&\lim_{t\rightarrow 0}\|v(t)\|_{3,\infty}=0,\\&
\sup_{0<\tau<\infty}\|v(\tau)\|_{3,\infty}
\leq C(\alpha+1)|a|\big\},
\end{align*}
where $C>0$ 
is independent of $a$ and $\psi$. 
\item Let $3<q<\infty$. Then there exists $\delta_2(q)\in(0,\delta_1]$ 
such that if $(\alpha+1)|a|\leq \delta_2$,
\begin{align*}
t^{\frac{1}{2}-\frac{3}{2q}}v
\in BC_{w^*}\big((0,\infty);L_{\sigma}^{q,\infty}(D)\big),
\end{align*}
where $v(t)$ is the solution obtained above. 
\end{enumerate}
\end{prop}
\begin{proof}
We first show the assertion 1 by the contraction mapping principle. 
Given $v\in X$, we define
\begin{align*}
\braket{(\Phi v)(t),\varphi}=\rm{the~RHS~of}~(\ref{NS6}),\quad 
\varphi\in C_{0,\sigma}^\infty(D).
\end{align*}
Lemma \ref{key} implies that $\Phi v\in X$ with
\begin{align}
&\|\Phi v\|_{X}\leq C_1\|v\|_{X}^2+C_2|a|\|v\|_{X}
+C_3(a^2+\alpha |a|),\label{thm1pf1}\\
&\|\Phi v-\Phi w\|_{X}
\leq (C_1\|v\|_{X}+C_1\|w\|_{X}+C_2|a|)\|v-w\|_{X}\label{thm1pf2}
\end{align}
for every $v,w\in X$. Here, 
$C_1,C_2,C_3,C_4$ are constants 
independent of $v,w,a$ and $\psi$. 
Hence, if we take $a$ satisfying  
\begin{align*}
(\alpha+1)|a|<\min
\left\{\frac{1}{2C_2},\frac{1}{16C_1C_3},\eta\right\}=:\delta_1,
\end{align*}
where $\eta\in (0,1]$ is a constant given in 
Proposition \ref{galdi2003}, 
then we obtain a unique solution $v$ within the class 
\begin{align*}
\{v\in X\mid\|v\|_X\leq 4C_3(\alpha+1)|a|\}
\end{align*}
which completes the proof of the assertion 1. 
\par We next show the assertion 2. 
By applying Lemma \ref{key}, 
we see that $\Phi v\in X_q$ together with 
(\ref{thm1pf1})--(\ref{thm1pf2}) in which 
$X$ norm was replaced by $X_q$ norm and 
the constants $C_i~(i=1,2,3)$ are also replaced by some others 
$\widetilde{C}_i(q)(\geq C_i)$. 
If we assume 
\begin{align}\label{acondiq}
(\alpha+1)|a|<\min\left\{\frac{1}{2\widetilde{C}_2},
\frac{1}{16\widetilde{C}_1\widetilde{C}_3}
,\eta\right\}=:\delta_2\,(\leq \delta_1),
\end{align}
we can obtain a unique solution $\hat{v}$ within the class 
\begin{align*}
\{v\in X_q\mid\|v\|_{X_q}
\leq 4\widetilde{C}_3(\alpha+1)|a|\}.
\end{align*}
Under the condition (\ref{acondiq}), 
let $v$ be the solution obtained in the assertion 1. Then we have 
(\ref{v1v22}) in which $v_1$, $v_2$ are replaced by $v$ and $\hat{v}$. 
By applying (\ref{e1}), we see that 
\begin{align*}
\|v-\hat{v}\|_X\leq 
\big\{C_1(\|v\|_X+\|\hat{v}\|_X)+C_2|a|\big\}\|v-\hat{v}\|_X
\leq
\big\{8\,\widetilde{C}_1\widetilde{C}_3(1+\alpha)|a|
+\widetilde{C}_2|a|\big\}\|v-\hat{v}\|_X.
\end{align*}
Furthermore, the condition (\ref{acondiq}) yields
\begin{align*}
8\,\widetilde{C}_1\widetilde{C}_3(1+\alpha)|a|+\widetilde{C}_2|a|<1
\end{align*}
which leads us to $v=\hat{v}$. The proof is complete.
\end{proof}
We note that Proposition \ref{thm2} implies
\begin{align}\label{cw*r}
t^{\frac{1}{2}-\frac{3}{2r}}v\in 
BC_{w}\big((0,\infty);L^r_{\sigma}(D)\big)
\end{align}
for all $r\in (3,q)$ by the interpolation inequality
\begin{align*}
\|f\|_r\leq C\|f\|_{3,\infty}^{1-\beta}\|f\|_{q,\infty}^\beta,
\quad 
\frac{1}{r}=\frac{1-\beta}{3}+\frac{\beta}{q},
\end{align*}
see (\ref{realinterpolation}). 
\par Let $q\in (6,\infty),$ then  
the solution obtained in Proposition \ref{thm2} 
also fulfills the following decay properties.
\begin{prop}\label{linftydecay}
\par Let $\psi$ be a function on $\R$ satisfying (\ref{psidef}) and 
we put $\alpha:=\displaystyle\max_{t\in\R}|\psi'(t)|$. 
Suppose that $6<q<\infty$. Then, under the same condition 
as in the latter part of 
Proposition \ref{thm2}, the solution $v$ obtained in 
Proposition \ref{thm2} satisfies $v(t)\in L^r(D)~(t>0)$ with 
\begin{align}\label{q<decay}
\|v(t)\|_{r}=O(t^{-\frac{1}{2}+\frac{3}{2q}})
\quad {\rm{as}}~t\rightarrow \infty
\end{align}
for $r\in (q,\infty]$. 
\end{prop}

\begin{proof}
We first show (\ref{q<decay}) with $r=\infty$, that is, 
$v(t)\in L^\infty(D)$ for $t>0$ with 
\begin{align}\label{inftydecay}
\|v(t)\|_{\infty}=O(t^{-\frac{1}{2}+\frac{3}{2q}})
\quad {\rm{as}}~t\rightarrow \infty. 
\end{align}
We note by continuity that $C^\infty_{0,\sigma}(D)$ can be 
replaced by $PC_0^{\infty}(D)$ 
as the class of test functions in (\ref{NS6}). 
Hence, it follows that 
\begin{align}\label{thm1pf6}
\sup_{\varphi\in C_0^{\infty}(D),\|\varphi\|_1\leq 1}|(v(t),\varphi)|
\leq N_1+N_2+N_3+N_4,
\end{align}
where
\begin{align*}
&N_1=\sup_{\varphi\in C_0^{\infty}(D),\|\varphi\|_1\leq 1}
\int_0^t |(v(\tau)\otimes v(\tau),\nabla T(t,\tau)^*P\varphi)|\,d\tau,\\
&N_2=\sup_{\varphi\in C_0^{\infty}(D),\|\varphi\|_1\leq 1}
\int_0^t |(v(\tau)\otimes u_s,\nabla T(t,\tau)^*P\varphi)|\,d\tau,\\
&N_3=\sup_{\varphi\in C_0^{\infty}(D),\|\varphi\|_1\leq 1}
\int_0^t |u_s\otimes v(\tau),\nabla T(t,\tau)^*P\varphi)|\,d\tau,\\
&N_4=
\sup_{\varphi\in C_0^{\infty}(D),\|\varphi\|_1\leq 1}
\int_0^t |(H(\tau),T(t,\tau)^*P\varphi)|\,d\tau.
\end{align*}
We begin by considering $N_1$. 
In view of (\ref{gradl1lrad}), we have
\begin{align*}
\int_0^t|(v(\tau)\otimes v(\tau),\nabla T(t,\tau)^*P\varphi)|\,d\tau
&\leq C[v]_{q,\infty}^2
\int_0^t\tau^{-1+\frac{3}{q}}
\left\|
\nabla T(t,\tau)^*P\varphi\right\|_{(1-\frac{2}{q})^{-1},1}\,d\tau\\
&\leq C[v]_{q,\infty}^2\int_0^t 
\tau^{-1+\frac{3}{q}}(t-\tau)^{-\frac{3}{q}-\frac{1}{2}}\,d\tau
\,\|\varphi\|_1\\&
\leq C[v]_{q,\infty}^2t^{-\frac{1}{2}}\|\varphi\|_1
\end{align*}
for all $\varphi\in C_0^{\infty}(D)$ and $t>0$. 
Here, the integrability is ensured because of $q\in (6,\infty)$. 
Hence we obtain 
\begin{align}\label{thm1pf7}
N_1\leq Ct^{-\frac{1}{2}}\quad {\rm{for}} ~t>0.
\end{align}
We next consider $N_2$. 
By applying (\ref{gradl1lrad}), it follows that 
\begin{align*}
\int_0^t|(v(\tau)\otimes u_s,\nabla T(t,\tau)^*P\varphi)|\,d\tau&\leq 
[v]_{q,\infty}\|u_s\|_{q,\infty}\int_0^t\tau^{-\frac{1}{2}+\frac{3}{2q}}
\|\nabla T(t,\tau)^*P\varphi\|_{(1-\frac{2}{q})^{-1},1}\,d\tau
\nonumber\\
&\leq C [v]_{q,\infty}\|u_s\|_{q,\infty}t^{-\frac{3}{2q}}\|\varphi\|_1
\end{align*}
for $t>0$. We thus have 
\begin{align}
N_2\leq Ct^{-\frac{3}{2q}}\quad {\rm{for}} ~t>0.
\end{align}
We next intend to derive the rate of decay $N_2$ 
as fast as possible. To this 
end, we split the integral into 
\begin{align}\label{split}
\int_0^t |(v(\tau)\otimes u_s,\nabla T(t,\tau)^*P\varphi)|\,d\tau
=\int_0^{\frac{t}{2}}+\int^{t-1}_{\frac{t}{2}}+\int_{t-1}^t
\end{align}
for $t>2$. We apply (\ref{gradl1lrad}) again to find 
\begin{align*}
\int_0^{\frac{t}{2}}&
\leq \|u_s\|_{3,\infty}\|v\|_X
\int_0^{\frac{t}{2}}\|\nabla T(t,\tau)^*P\varphi\|_{3,1}\,d\tau
\leq Ct^{-\frac{1}{2}}\|\varphi\|_1,\\
\int_{\frac{t}{2}}^{t-1}&\leq\|u_s\|_{3,\infty}[v]_{q,\infty}
\int_{\frac{t}{2}}^{t-1}\tau^{-\frac{1}{2}+\frac{3}{2q}}
\|\nabla T(t,\tau)^*P\varphi\|_{(1-\frac{1}{3}-\frac{1}{q})^{-1},1}\,d\tau
\leq Ct^{-\frac{1}{2}+\frac{3}{2q}}\|\varphi\|_1
\end{align*}
and
\begin{align*}
\int_{t-1}^{t}\leq \|u_s\|_{q,\infty}[v]_{q,\infty}
\int_{t-1}^t\tau^{-\frac{1}{2}+\frac{3}{2q}}\|
\nabla T(t,\tau)^*P\varphi\|_{(1-\frac{2}{q})^{-1},1}\,d\tau
\leq Ct^{-\frac{1}{2}+\frac{3}{2q}}\|\varphi\|_1
\end{align*}
for all $\varphi\in C_0^{\infty}(D)$ and $t>2$. 
Summing up the estimates above, we are led to 
\begin{align}\label{thm1pf8}
N_2\leq Ct^{-\frac{1}{2}+\frac{3}{2q}}\quad {\rm{for}} ~t>2.
\end{align}
Similarly, we have
\begin{align}
&N_3\leq Ct^{-\frac{3}{2q}}\quad {\rm{for}} ~t>0,\\
&N_3\leq Ct^{-\frac{1}{2}+\frac{3}{2q}}\quad {\rm{for}} ~t>2.
\end{align}
It is easily seen from (\ref{psidef}), 
(\ref{hstaest2}) and (\ref{l1lr2ad}) that 
\begin{align*}
\int_0^t |(H(\tau),T(t,\tau)^*P\varphi)|\,d\tau
&\leq 
C(a^2+\alpha|a|)\int_0^{\min\{1,t\}}
\|T(t,\tau)^*P\varphi\|_{\frac{3}{2},1}\,d\tau\\
&\leq C(a^2+\alpha|a|)\int_0^{\min\{1,t\}} 
(t-\tau)^{-\frac{1}{2}}\,d\tau\,\|\varphi\|_1
\end{align*} 
for all $\varphi\in C_0^{\infty}(D)$ and $t>0$, which yields 
\begin{align} 
&N_4\leq Ct^{\frac{1}{2}}\quad {\rm{for}} ~t>0,\\
&N_4\leq Ct^{-\frac{1}{2}}\quad {\rm{for}} ~t>2.\label{thm1pf9}
\end{align}
Combining (\ref{thm1pf6})--(\ref{thm1pf9}) 
implies $v(t)\in L^\infty(D)$ ($t>0$) and 
(\ref{inftydecay}). In view of the interpolation relation 
\begin{align*}
(L^{q,\infty}(D),L^{\infty}(D))_{1-\frac{q}{r},r}=L^r(D),\quad 
\|f\|_r\leq C\|f\|^{\frac{q}{r}}_{q,\infty}
\|f\|^{1-\frac{q}{r}}_\infty,\quad q<r<\infty, 
\end{align*}
we obtain (\ref{q<decay}) for $r\in (q,\infty]$ as well. 
This completes the proof.
\end{proof}

\begin{rmk}\label{rmk3}
When the stationary solution possesses the scale-critical rate 
$O(1/|x|)$, the $L^\infty$ decay of perturbation with less sharp 
rate $O(t^{-\frac{1}{2}+\varepsilon})$ was derived first by 
Koba \cite{koba2017} in the context of stability analysis, where 
$\varepsilon>0$ is arbitrary. If we have a look only at the 
$L^\infty$ decay rate, our rate is comparable with his result 
since $q\in (6,\infty)$ is arbitrary. However, we are not able to 
prove Proposition \ref{linftydecay} by his method. 
This is because he doesn't split the integrals in $N_2$ and $N_3$, 
so that the rate of $L^\infty$ decay is slower than the one of 
$L^{q,\infty}$ decay. From this point of view, Proposition 
\ref{linftydecay} is regarded as a slight improvement of his result. 
\end{rmk}

We next show that the solution $v$ obtained in Proposition \ref{thm2} 
actually satisfies the integral equation 
(\ref{integraleq}) by 
identifying $v$ with a local solution $\widetilde{v}$ of (\ref{integraleq}) 
in a neighborhood of each time $t>0$. 
To this end, we need the following lemma on the uniqueness. 
The proof is similar to the argument in the second half 
(after (\ref{v1v22})) of the proof of Theorem \ref{thm3} and thus we may 
omit it.

\begin{lem}\label{uniqueness}
Let $3<r<\infty,~0\leq t_0<t_1<\infty$ 
and $v_0\in L^r_{\sigma}(D)$. Then the problem
\begin{align}\label{NS7}
(v(t),\varphi)=(v_0,T(t,t_0)^*\varphi)
&+\int_{t_0}^t\big(v(\tau)\otimes v(\tau)
+\psi(\tau)\{v(\tau)\otimes u_s+u_s\otimes v(\tau)\},
\nabla T(t,\tau)^*\varphi\big)\,d\tau\nonumber\\&
+\int_{t_0}^t(H(\tau),T(t,\tau)^*\varphi)\,d\tau,\quad 
\forall \varphi\in C^{\infty}_{0,\sigma}(D)
\end{align}
on $(t_0,t_1)$ admits at most one solution within the class 
$L^{\infty}(t_0,t_1;L_{\sigma}^r(D))$. 
Here, $H$ is given by (\ref{h}). 
\end{lem}

Given $v_0\in L^r_{\sigma}(D)$ with $r\in (3,\infty)$, let us 
construct a local solution of the integral equation
\begin{align}\label{intt0}
v(t)=T(t,t_0)v_0+\int_{t_0}^t &T(t,\tau)P[(Gv)(\tau)+H(\tau)]\,d\tau,
\end{align} 
where $G$ and $H$ are defined by (\ref{g}) and (\ref{h}), 
respectively. For $0\leq t_0<t_1<\infty$ and $r\in (3,\infty)$, 
we define the function space
\begin{align}
&Y_r(t_0,t_1)=\big\{v\in C\big([t_0,t_1];L^{r}_{\sigma}(D)\big)
\,\big|\,
(\cdot-t_0)^{\frac{1}{2}}\nabla v(\cdot)\in 
BC_w\big((t_0,t_1];L^r(D)\big)\big\}
\end{align}
which is a Banach space equipped with norm 
\begin{align}
&\|v\|_{Y_r(t_0,t_1)}=\sup_{t_0\leq\tau\leq t_1}\|v(\tau)\|_r
+\sup_{t_0<\tau\leq t_1}
(\tau-t_0)^{\frac{1}{2}}\|\nabla v(\tau)\|_r\label{Yr}
\end{align}
and set 
\begin{align}
&U_1(v,w)(t)=\int_{t_0}^t 
T(t,\tau)P[v(\tau)\cdot \nabla w(\tau)]\,d\tau,\quad 
U_2(v)(t)=\int_{t_0}^t T(t,\tau)P[\psi(\tau)v(\tau)\cdot \nabla u_s]\,d\tau,\nonumber\\
&U_3(v)(t)=\int_{t_0}^t T(t,\tau)
P[\psi(\tau)u_s\cdot \nabla v(\tau)]\,d\tau,\quad 
U_4(t)=\int_{t_0}^t T(t,\tau)PH(\tau)\,d\tau.\label{Udef}
\end{align}

\begin{lem}\label{localkey}
Let $3<r<\infty$ and $0\leq t_0<t_1\leq t_0+1$. 
Suppose that $v,w\in Y_r(t_0,t_1)$. 
Then $U_1(v,w),U_2(v),U_3(v),U_4\in Y_r(t_0,t_1)$. 
Furthermore, there exists a constant $C=C(r,t_0)$ such that 
\begin{align}
&\|U_1(v,w)\|_{Y_r(t_0,t)}\leq C(t-t_0)^{\frac{1}{2}-\frac{3}{2r}}
\|v\|_{Y_r(t_0,t)}\|w\|_{Y_r(t_0,t)\,,}\label{Y1}\\
&\|U_2(v)\|_{Y_r(t_0,t)}\leq 
C(t-t_0)^{1-\frac{3}{2r}}\|\nabla u_s\|_r
\|v\|_{Y_r(t_0,t)\,,}\label{Y2}\\
&\|U_3(v)\|_{Y_r(t_0,t)}\leq C(t-t_0)^{\frac{1}{2}-\frac{3}{2r}}
\|u_s\|_r\|v\|_{Y_r(t_0,t)\,,}\label{Y3}\\
&\|U_4\|_{Y_r(t_0,t)}
\leq C(t-t_0)^{\frac{1}{2}+\frac{3}{2r}}(a^2+\alpha|a|)\label{Y4}
\end{align}
for all $t\in (t_0,t_1].$
\end{lem}

\begin{proof}
In view of (\ref{lqlr0}), we have 
\begin{align}
\|U_1(t)\|_r\leq C\int_{t_0}^t (t-\tau)^{-\frac{3}{2r}}
\|v\|_r\|\nabla w\|_r\,d\tau
\leq C(t-t_0)^{-\frac{3}{2r}+\frac{1}{2}}
\|v\|_{Y_r(t_0,t)}\|w\|_{Y_r(t_0,t)}.\label{Y11}
\end{align}
Furthermore, (\ref{gradT}) with $\T=t_0+1$ yields 
\begin{align}\label{Y12}
\|\nabla U_1(t)\|_r
\leq C(t-t_0)^{-\frac{3}{2r}}\|v\|_{Y_r(t_0,t)}
\|w\|_{Y_r(t_0,t)}.
\end{align}
By (\ref{Y11}) and (\ref{Y12}), we obtain (\ref{Y1}). 
Similarly, we can show (\ref{Y2})--(\ref{Y4}). 
We note that the estimate (\ref{Y4}) 
follows from (\ref{gradT2}) with $\T=t_0+1$ 
together with (\ref{hstaest2}).
\par We next show the continuity of $U_1$ with respect to $t$.  
Let $t_2\in [t_0,t_1]$. If $t_2<t,$ we have  
\begin{align*}
U_1(t)-U_1(t_2)&=\int_{t_0}^{t_2}(T(t,t_2)-1)
T(t_2,\tau)P[v(\tau)\cdot\nabla w(\tau)]\,d\tau
+\int_{t_2}^t T(t,\tau)P[v(\tau)\cdot
\nabla w(\tau)]\,d\tau\\&=:U_{11}(t)+U_{12}(t).
\end{align*}
Lebesgue's convergence theorem yields that 
$\|U_{11}(t)\|_r\rightarrow 0$ as $t\rightarrow t_2$, while 
\begin{align*}
\|U_{12}(t)\|_r\leq C(t-t_2)^{\frac{1}{2}-\frac{3}{2r}}
\|v\|_{Y_r(t_0,t_1)}\|w\|_{Y_r(t_0,t_1)}
\rightarrow 0 \quad {\rm{as}}~t\rightarrow t_2.
\end{align*}
To discuss the case $t<t_2$, we need the following device. 
Let $(t_0+t_2)/2\leq \tilde{t}<t<t_2$, 
where $\tilde{t}$ will be determined later, then  
\begin{align*}
U_1(t)-U_1(t_2)=\left(\int_{t_0}^{\tilde{t}}+\int_{\tilde{t}}^t\right)
T(t,\tau)&P[v(\tau)\cdot \nabla w(\tau)]\,d\tau\\
&-\left(\int_{t_0}^{\tilde{t}}+\int_{\tilde{t}}^{t_2}\right)
T(t_2,\tau)P[v(\tau)\cdot \nabla w(\tau)]\,d\tau.
\end{align*}
We observe that 
\begin{align*}
\int_{\tilde{t}}^t&\|T(t,\tau)P[v(\tau)\cdot\nabla w(\tau)]\|_r\,d\tau
+\int_{\tilde{t}}^{t_2}
\|T(t_2,\tau)P[v(\tau)\cdot\nabla w(\tau)]\|_r\,d\tau\\&
\leq C\|v\|_{Y_r(t_0,t_1)}\|w\|_{Y_r(t_0,t_1)}\left(
\int_{\tilde{t}}^t(t-\tau)^{-\frac{3}{2r}}
(\tau-t_0)^{-\frac{1}{2}}\,d\tau+
\int_{\tilde{t}}^{t_2}(t_2-\tau)^{-\frac{3}{2r}}
(\tau-t_0)^{-\frac{1}{2}}\,d\tau\right)\\
&\leq\frac{2C}{1-\frac{3}{2r}}\|v\|_{Y_r(t_0,t_1)}\|w\|_{Y_r(t_0,t_1)}
\left(\frac{t_0+t_2}{2}-t_0\right)^{-\frac{1}{2}}
(t_2-\tilde{t}\,)^{1-\frac{3}{2r}}.
\end{align*}
For any $\varepsilon>0$, we choose $\tilde{t}$ such that 
\begin{align*}
\frac{2C}{1-\frac{3}{2r}}\|v\|_{Y_r(t_0,t_1)}\|w\|_{Y_r(t_0,t_1)}
\left(\frac{t_0+t_2}{2}-t_0\right)^{-\frac{1}{2}}
(t_2-\tilde{t}\,)^{1-\frac{3}{2r}}<\varepsilon
\end{align*}
which yields
\begin{align*}
\|U_1(t)-U_1(t_2)\|_r&\leq 
\int_{t_0}^{\tilde{t}}
\left\|\big(T(t,\tau)-T(t_2,\tau)\big)
P[v(\tau)\cdot\nabla w(\tau)]\right\|_r\,d\tau+\varepsilon
\quad {\rm{for~}}\tilde{t}<t<t_{2}
\end{align*}
and therefore, 
\begin{align}\label{111}
\limsup_{t\rightarrow t_2}\|U_1(t)-U_1(t_2)\|_r
\leq \limsup_{t\rightarrow t_2}\int_{t_0}^{\tilde{t}}
\left\|\big(T(t,\tau)-T(t_2,\tau)\big)
P[v(\tau)\cdot\nabla w(\tau)]\right\|_r\,d\tau+\varepsilon.
\end{align}
Since 
$\big\|\big(T(t,\tau)-T(t_2,\tau)\big)
P[v(\tau)\cdot\nabla w(\tau)]\big\|_r=
\big\|\big(T(t,\tilde{t}\,)-T(t_2,\tilde{t}\,)\big) 
T(\tilde{t},\tau)P[v(\tau)\cdot\nabla w(\tau)]\big\|_r$
tends to $0$ as 
$t\rightarrow t_2$ for $t_0<\tau<\tilde{t}$, 
it follows from Lebesgue's convergence theorem 
that the integral term in (\ref{111}) tends to $0$ as $t\rightarrow t_{2}$. 
Since $\varepsilon>0$ is arbitrary, we have
\begin{align}\label{Y1conti}
U_1\in C\big([t_0,t_1];L^r_{\sigma}(D)\big).
\end{align}
Furthermore, we find $\nabla U_1\in C_w\big((t_0,t_1];L^r(D)\big)$ 
on account of (\ref{Y12}) and (\ref{Y1conti}) 
together with the relation  
\begin{align*}
(\nabla U_1(t)-\nabla U_1(t_2),\varphi)
=-(U_1(t)-U_1(t_2),\nabla \cdot\varphi)
\end{align*}
for all $t_2\in (t_0,t_1]$ and 
$\varphi\in C_0^{\infty}(D)^{3\times 3}$. 
Since $U_2,U_3$ and $U_4$ are discussed similarly, 
the proof is complete.
\end{proof}

The following proposition provides a local solution of (\ref{intt0}).
\begin{prop}\label{localexistence}
Let $3<r<\infty$, $t_0\geq 0$ and $v_0\in L^r_{\sigma}(D)$. 
There exists $t_1\in (t_0,t_0+1]$ such that (\ref{intt0}) 
admits a unique solution $v\in Y_r(t_0,t_1)$.
Moreover, the length of the existence interval can be estimated 
from below by
\begin{align*}
t_1-t_0\geq \zeta(\|v_0\|_r),
\end{align*}
where $\zeta(\cdot):[0,\infty)\rightarrow (0,1)$ 
is a non-increasing function defined by (\ref{zeta}) below.
\end{prop}

\begin{proof}
We put
\begin{align*}
(\Psi v)(t)={\rm{the~RHS~of~(\ref{intt0})}}.
\end{align*} 
By applying Lemma \ref{localkey}, we have 
\begin{align*}
&\|\Psi v\|_{Y_r(t_0,t)}\leq 
(C_1\|v\|_{Y_r(t_0,t)}^2
+C_2\|v\|_{Y_r(t_0,t)}
+C_3)(t-t_0)^{\frac{1}{2}-\frac{3}{2r}}+C_4\|v_0\|_r,\\
&\|\Psi v-\Psi w\|_{Y_r(t_0,t)}\leq
\{C_1(\|v\|_{Y_r(t_0,t)}
+\|w\|_{Y_r(t_0,t)})+C_2\}(t-t_0)^{\frac{1}{2}-\frac{3}{2r}}
\|v-w\|_{Y_r(t_0,t)}
\end{align*}
for all $t\in(t_0,t_0+1]$ and $v,w\in Y_r(t_0,t)$. 
We note that the constants $C_i$ may be 
dependent on $\|u_s\|_r$, $\|\nabla u_s\|_r$, $\alpha$ and $a$. 
We choose $t_1\in(t_0,t_0+1]$ such that 
\begin{align}
&C_2(t_1-t_0)^{\frac{1}{2}-\frac{3}{2r}}<\frac{1}{2},\label{t12}\\
&8C_1(t_1-t_0)^{\frac{1}{2}-\frac{3}{2r}}
\{C_3(t_1-t_0)^{\frac{1}{2}-\frac{3}{2r}}+C_4\|v_0\|_r\}
<\frac{1}{2}\label{t13}
\end{align}
which imply
\begin{align}\label{t11}
&\lambda:=\big\{1-C_2(t_1-t_0)^{\frac{1}{2}-\frac{3}{2r}}\big\}^2
-4C_1(t_1-t_0)^{\frac{1}{2}-\frac{3}{2r}}
\big\{C_3(t_1-t_0)^{\frac{1}{2}-\frac{3}{2r}}+C_4\|v_0\|_r\big\}>0.
\end{align}
We set  
\begin{align*}
&\Lambda:=\frac{1-C_2(t_1-t_0)^{\frac{1}{2}-\frac{3}{2r}}
-\sqrt{\lambda}}{2C_1(t_1-t_0)^{\frac{1}{2}-\frac{3}{2r}}}
<4(C_3+C_4\|v_0\|_r),\\ 
&Y_{r,\Lambda}(t_0,t_1):=\{v\in Y_r(t_0,t_1)\mid\|v\|_{Y_r(t_0,t_1)}
\leq \Lambda\}.
\end{align*}
Then we find that the map $\Psi:Y_{r,\Lambda}(t_0,t_1)\rightarrow 
Y_{r,\Lambda}(t_0,t_1)$ is well-defined and also contractive. 
Hence we obtain a local solution. 
Indeed, the conditions (\ref{t12}) and (\ref{t13}) are accomplished by 
\begin{align*} 
t_1-t_0<\min\left\{1,\left(\frac{1}{2C_2}\right)
^{(\frac{1}{2}-\frac{3}{2r})^{-1}},
\left(\frac{1}{16C_1(C_3+C_4\|v_0\|_r)}\right)
^{(\frac{1}{2}-\frac{3}{2r})^{-1}}\right\}.
\end{align*}
Thus, it is possible to take $t_1$ such that 
\begin{align}\label{zeta}
t_1-t_0\geq
\frac{1}{2}\min\left\{1,\left(\frac{1}{2C_2}\right)
^{(\frac{1}{2}-\frac{3}{2r})^{-1}},
\left(\frac{1}{16C_1(C_3+C_4\|v_0\|_r)}\right)
^{(\frac{1}{2}-\frac{3}{2r})^{-1}}\right\}=:\zeta(\|v_0\|_r).
\end{align}
The proof is complete. 
\end{proof}

\begin{lem}\label{betterregular} 
Let $3<r<\infty$, $t_0\geq 0$ and $v_0\in L^r_\sigma(D).$ 
The local solution $v$ obtained in Proposition \ref{localexistence} 
also possesses the following properties:
\begin{align}\label{regular}
v\in C\big((t_0,t_1];L^\kappa_\sigma(D)\big)
\cap C_{w^*}\big((t_0,t_1];L^\infty(D)\big)
\end{align}
for every $\kappa\in (r,\infty)$ and 
\begin{align}
\nabla v\in C_w\big((t_0,t_1];L^\gamma(D)\big) 
\end{align}
for every $\gamma\in(r,\infty)$ satisfying
\begin{align}\label{rlrelation}
\frac{2}{r}-\frac{1}{\gamma}<\frac{1}{3}.
\end{align}
\end{lem}

\begin{proof}
By using (\ref{lqlr0}) and (\ref{lqlr2}) and the semigroup property, 
we find $v(t)\in L^{\infty}(D)$ with 
\begin{align}\label{inftybdd}
\|v(t)\|_\infty\leq C(t-t_0)^{-\frac{3}{2r}}\big\{\|v_0\|_r+
\|v\|^2_{Y_r(t_0,t_1)}+
\|v\|_{Y_r(t_0,t_1)}(\|u_s\|_r+\|\nabla u_s\|_r)+(a^2+\alpha|a|)
\big\}
\end{align}
for all $t\in (t_0,t_1]$. Moreover, for each $t_2\in(t_0,t_1]$, 
we know from $v\in C([t_0,t_1];L^r_\sigma(D))$ that 
\begin{align}
\big(v(t),\varphi\big)-\big(v(t_2),\varphi\big)
\rightarrow 0\quad {\rm as}~t\rightarrow t_2
\end{align}
for all $\varphi\in C_0^{\infty}(D)$, which combined with (\ref{inftybdd}) 
yields $v\in C_{w^*}\big((t_0,t_1];L^\infty(D)\big)$. Since
\begin{align*}
\|v(t)-v(t_2)\|_\kappa\leq \|v(t)-v(t_2)\|^{\frac{r}{\kappa}}_r
\|v(t)-v(t_2)\|_\infty^{1-\frac{r}{\kappa}}
\end{align*}
for $\kappa\in(r,\infty)$ and $t_2\in(t_0,t_1]$, 
it follows from (\ref{inftybdd}) that  
\begin{align}\label{contikappa1}
v\in C\big((t_0,t_1];L_\sigma^\kappa(D)\big)
\quad{\rm for}~\kappa\in (r,\infty).
\end{align} 
\par The estimates (\ref{gradT}) and (\ref{gradT2}) with $\T=t_0+1$ 
imply that if we assume (\ref{rlrelation}), 
we have $\nabla v(t)\in L^\gamma (D)$ with 
\begin{align}
\|\nabla v(t)\|_\gamma\leq
C(t-t_0)^{-\frac{3}{2}(\frac{1}{r}-\frac{1}{\gamma})-\frac{1}{2}}\big\{
\|v_0\|_r+\|v\|^2_{Y_r(t_0,t_1)}
&+\|v\|_{Y_r(t_0,t_1)}(\|u_s\|_r+\|\nabla u_s\|_r)\nonumber\\&
+(a^2+\alpha|a|)\big\}\label{nablabdd}
\end{align}
for all $t\in (t_0,t_1]$. 
Here, we note that (\ref{rlrelation}) is needed for estimates of  
$\nabla U_1$ and $\nabla U_3$ given in (\ref{Udef}). 
On account of (\ref{contikappa1}), (\ref{nablabdd}) and 
\begin{align*}
(\nabla v(t)-\nabla v(t_2),\varphi)
=-(v(t)-v(t_2),\nabla\cdot\varphi)
\end{align*}
for all $t_2\in (t_0,t_1]$ and $\varphi\in C_0^\infty(D)^{3\times 3},$ 
we find the weak continuity of $\nabla v$ with values in $L^{\gamma}(D)$. 
The proof is complete. 
\end{proof}

We close the paper with completion of the proof of Theorem \ref{thm1}. 
\vspace{0.2cm}\\
\noindent{\bf Proof of Theorem \ref{thm1}}\quad It remains to show that 
the solution $v$ obtained in Proposition \ref{thm2} 
also satisfies (\ref{integraleq}) with 
\begin{align}\label{contikappa}
v\in C\big((0,\infty);L^\kappa_\sigma(D)\big)\cap 
C_{w^*}\big((0,\infty);L^\infty(D)\big),\quad 
\nabla v\in C_w\big((0,\infty);L^\kappa(D)\big)
\end{align}
for all $3<\kappa<\infty$. 
Let $t_*\in (0,\infty).$ 
By applying Proposition \ref{localexistence} 
and Lemma \ref{betterregular} with $r=6$, 
we can see that for each $t_0\in [t_*/2,t_*)$, 
there exists $\widetilde{v}\in Y_6(t_0,t_1)$ 
which satisfies (\ref{intt0}) and therefore, (\ref{NS7}) with $v_0=v(t_0)$ 
such that
\begin{align*}
\widetilde{v}\in C\big((t_0,t_1];L^\kappa_\sigma(D)\big)\cap 
C_{w^*}\big((t_0,t_1];L^\infty(D)\big),\quad 
\nabla\, \widetilde{v}\in C_w\big((t_0,t_1];L^\kappa(D)\big)
\end{align*}
for all $\kappa\in [6,\infty)$. 
Moreover, the length of the existence interval can be estimated by 
\begin{align*}
t_1-t_0\geq \zeta(\|v(t_0)\|_6)\geq 
\zeta\left(C_5\left(\frac{t_*}{2}\right)^{-\frac{1}{4}}\right)
=:\varepsilon,
\end{align*}
where $\zeta(\cdot)$ is the non-increasing function 
in Proposition \ref{localexistence} because of 
\begin{align*}
\|v(t)\|_6\leq C_5\left(\frac{t_*}{2}\right)^{-\frac{1}{4}}
\end{align*}
for all $t\geq t_*/2$, see (\ref{cw*r}). 
We note that the solution $v$ obtained in Proposition \ref{thm2} 
also satisfies (\ref{NS7})  with $v_0=v(t_0)$ 
since $C^\infty_{0,\sigma}(D)$ can be replaced 
by $L^{6/5}_\sigma(D)$ as the class of test functions in (\ref{NS6}).  
Let us take $t_0:=\max\{t_*/2,t_*-\varepsilon/2\}$ 
so that $t_*\in (t_0,t_1)$, 
in which $v=\widetilde{v}$ on account of Lemma \ref{uniqueness}.  
Since $t_*$ is arbitrary, 
we conclude (\ref{contikappa}) for $\kappa\in [6,\infty)$. 
It is also proved by applying 
Proposition \ref{localexistence} with $r\in (3,6)$ 
that the solution belongs 
to the class (\ref{contikappa}) for $\kappa\in (3,6)$ as well.  
The proof is complete.\qed
\vspace{0.3cm}\\
\noindent{\bf{Acknowledgment.}}~~The author would like to 
thank Professor Toshiaki Hishida for valuable 
discussions to improve this paper.

\begin{align*}
\hspace{8cm}&\rm numa~adreso:\\
&\rm Graduate~School~of~Mathematics\\
&\rm Nagoya ~University\\
&\rm Furo\mbox{-}cho, Chikusa\mbox{-}ku\\
&\rm Nagoya, 464\mbox{-}8602 \\
&\rm Japan\\
&\rm E\mbox{-}mail: m17023c@math.nagoya\mbox{-}u.ac.jp
\end{align*}


\begin{thebibliography}{99}
\bibitem{belobook1976} 
Bergh, J. and L\"{o}fstr\"{o}m, J., 
Interpolation Spaces, Springer, Berlin, 1976.






\bibitem{bomi1995}

Borchers, W. and Miyakawa, T., 
On stability of exterior stationary Navier-Stokes flows, 
Acta Math. {\bf{174}} (1995), 311--382. 








\bibitem{fahi2007}
Farwig, R. and Hishida, T., 
Stationary Navier-Stokes flow around a rotating obstacle, 
Funkcial. Ekvac. {\bf{50}} (2007), 371--403.








\bibitem{fane2007}
Farwig, R., and Neustupa, J., 
On the spectrum of a Stokes-type operator 
arising from flow around a rotating body, 
Manuscripta Math. {\bf{122}} (2007), 419--437.


\bibitem{finn1965}
Finn, R., Stationary solutions of the Navier-Stokes equations, 
Proc, Symp. Appl. Math. {\bf{17}} (1965), 121--153.


\bibitem{fuka1964}
Fujita. H. and Kato, T., On the Navier-Stokes initial value problem. I, 
Arch. Rational Mech. Anal. {\bf{16}} (1964), 269--315.




\bibitem{fumo1977}
Fujiwara, D. and Morimoto, H., 
An $L_r$-theorem of the Helmholtz decomposition of vector fields, 
J. Fac. Sci. Univ. Tokyo Sect. IA Math {\bf{24}} (1977), 685--700. 





\bibitem{galdi2003}
Galdi, G.P., Steady flow of a Navier-Stokes fluid 
around a rotating obstacle, J. Elast. {\bf{71}} (2003), 1--31.



\bibitem{gahesh1997}
Galdi, G.P., Heywood, J.G. and Shibata, Y., 
On the global existence and convergence to steady state of 
Navier-Stokes flow past an obstacle that is started from rest, 
Arch. Rational Mech. Anal. {\bf{138}} (1997), 307--318.





\bibitem{harh2014}
Hansel, T. and Rhandi, A., 
The Oseen-Navier-Stokes flow in the exterior of a rotating obstacle: 
the non-autonomous case, J. Reine Angew. Math. {\bf{694}} (2014), 1--26.



\bibitem{heywood1972}
Heywood, J.G., 
The exterior nonstationary problem for the Navier-Stokes equations, 
Acta Math. {\bf{129}} (1972), 11--34



\bibitem{hishida1999}
Hishida, T., An existence theorem for the Navier-Stokes flow 
in the exterior of a rotating obstacle, 
Arch. Rational Mech. Anal. {\bf{150}} (1999), 307--348.






\bibitem{hishida2013}
Hishida, T., Mathematical analysis of the equations 
for incompressible viscous fluid around a rotating obstacle, 
Sugaku Expositions, {\bf{26}} (2013), 149--179.





\bibitem{hishida2018}
Hishida, T., Large time behavior of a generalized Oseen evolution operator, 
with applications to the Navier-Stokes flow past a rotating obstacle, 
Math. Ann. {\bf{372}} (2018), 915--949.


\bibitem{hishidapre}
Hishida, T., Decay estimates of gradient of 
a generalized Oseen evolution operator arising from 
time-dependent rigid motions in exterior domains, 
Arch. Rational Mech. Anal. {\bf 238} (2020), 215--254.


\bibitem{hish2009}
Hishida, T., and Shibata, Y., 
$L_p$-$L_q$ estimate of the Stokes operator and Navier-Stokes flows 
in the exterior of a rotating obstacle, Arch. Rational Mech. Anal. 
{\bf{193}} (2009), 339--421.

\bibitem{kato1984}


Kato, T., Strong $L^p$ solutions of the Navier-Stokes equation in $\R^m$, 
with applications to weak solutions, 
Math. Z. {\bf{187}} (1984), 471--480.



\bibitem{koba2017}
Koba, H., On $L^{3,\infty}$-stability of 
the Navier-Stokes system in exterior domains, J. Differential Equations 
{\bf{262}} (2017), 2618--2683.


\bibitem{kosh1998}
Kobayashi, T. and Shibata, Y., 
On the Oseen equation in the three dimensional exterior domains, 
Math. Ann. {\bf{310}} (1998), 1--45.









\bibitem{koya1998}
Kozono, H., and Yamazaki, M., On a larger class of 
stable solutions to the Navier-Stokes equations in 
exterior domains, Math. Z., {\bf{228}} (1998), 751--785.


\bibitem{miyakawa1982}
Miyakawa, T., On nonstationary solutions of the Navier-Stokes equations in an exterior domain, 
Hiroshima Math. J., {\bf{12}} (1982), 115--140.



\bibitem{siso1992}
Simader, C.G. and Sohr, H., A new approach to the Helmholtz decomposition and 
the Neumann problem in $L^q$-spaces for bounded and exterior domains, 
Mathematical Problems Relating to the Navier-Stokes Equations (eds. G.P. Galdi), 1--35, 
Ser. Adv. Math. Appl. Sci., {\bf{11}}, World Sci. Publ., River Edge, NJ, 1992.


\bibitem{yamazaki2000}
Yamazaki, M., The Navier-Stokes equations in the weak-$L^n$ space with 
time-dependent external force, 
Math. Ann., {\bf{317}} (2000), 635--675.















\end{thebibliography}
\end{document}